\documentclass[11pt]{amsart}

\usepackage{latexsym, amssymb,enumerate}
\usepackage{appendix}

\newcommand{\be}{\begin{equation}}
\newcommand{\ee}{\end{equation}}
\newcommand{\beq}{\begin{eqnarray}}
\newcommand{\eeq}{\end{eqnarray}}
\newcommand{\cM}{\mathcal{M}}

\def\R{{\mathfrak R}}

\newtheorem{prop}{Proposition}[section]
\newtheorem{theo}[prop]{Theorem}
\newtheorem{lemm}[prop]{Lemma}

\newtheorem{rema}[prop]{Remark}
\newtheorem{exam}[prop]{Example}
\newtheorem{defi}[prop]{Definition}

\def\begeq{\begin{equation}}
\def\endeq{\end{equation}}

\def\R{\mathbb R}

\def\tr{{\rm tr}}

\def\d{\delta}
\def\a{\alpha}

\def\s{\sigma}

\def \ds{\displaystyle}
\def \vs{\vspace*{0.1cm}}

\def\odot{\setbox0=\hbox{$\bigcirc$}\relax \mathbin {\hbox
to0pt{\raise.5pt\hbox to\wd0{\hfil $\wedge$\hfil}\hss}\box0 }}

\numberwithin{equation} {section}

\begin{document}

\title[Gauss-Bonnet-Chern mass] {A new mass for asymptotically flat manifolds  }

\author{Yuxin Ge}
\address{Laboratoire d'Analyse et de Math\'ematiques Appliqu\'ees,
CNRS UMR 8050,
D\'epartement de Math\'ematiques,
Universit\'e Paris Est-Cr\'eteil Val de Marne, \\61 avenue du G\'en\'eral de Gaulle,
94010 Cr\'eteil Cedex, France}
\email{ge@u-pec.fr}
\author{Guofang Wang}
\address{ Albert-Ludwigs-Universit\"at Freiburg,
Mathematisches Institut
Eckerstr. 1
D-79104 Freiburg}
\email{guofang.wang@math.uni-freiburg.de}

\author{Jie Wu}
\address{School of Mathematical Sciences, University of Science and Technology
of China Hefei 230026, P. R. China
\and
 Albert-Ludwigs-Universit\"at Freiburg,
Mathematisches Institut
Eckerstr. 1
D-79104 Freiburg
}
\email{jie.wu@math.uni-freiburg.de}
\thanks{The project is partly supported by SFB/TR71
``Geometric partial differential equations''  of DFG. The first author  is partly supported by ANR  project
ANR-08-BLAN-0335-01.}
\maketitle

\section{Introduction}

One of the important results in differential geometry is the Riemannian positive mass theorem (PMT): {\it Any asymptotically flat Riemannian manifold $\cM^n$ with a suitable decay order
and with nonnegative scalar curvature
has  nonnegative ADM mass.} 
 This theorem was proved by Schoen and Yau \cite{SY1} for manifolds of dimension $n\le 7$ using
a minimal hypersurface argument and by Witten \cite{Wit} for any spin manifold. See also \cite{PT}.
 For locally conformally flat manifolds the proof was given in \cite{SY2}
using the developing map.
Very recently, the PMT was proved by  by Lam \cite{Lam} for any  asymptotically flat Riemannian manifold $\cM^n$ which is represented by a graph in $\R^{n+1}$.
For general  higher dimensional manifolds, the proof of the  positive mass theorem was announced by Lohkamp \cite{Loh} by an argument extending the minimal hypersurface argument of Schoen and Yau and  by Schoen in \cite{Schoen2}. There are many generalizations of the positive mass theorem.
For example, a refinement of the PMT, the Riemannian Penrose inequality was proved by Huisken-Ilmanen \cite{HI} and Bray \cite{BP} for $n=3$
and by Bray and Lee \cite{BL} for  $n\le 7$. See also the excellent surveys  on the Riemannian Penrose inequality of Bray \cite{Bray} and Mars \cite{Mars}.

The ADM mass was introduced by Arnowitt,  Deser, and Misner \cite{ADM}
for asymptotically flat Riemannian manifolds. A complete  manifold $(\cM^n,g)$ is said to be an asymptotically flat (AF) of order $\tau$ (with one end) if there is a compact set $K$ such that
$\cM\setminus K$ is diffeomorphic to  $\mathbb{R}^n\setminus B_R(0)$ for some $R>0$ and in the standard coordinates in $\mathbb{R}^n$, the metric $g$ has the following expansion
 $$g_{ij}=\delta_{ij}+\sigma_{ij},$$
 with $$|\sigma_{ij}|+ r|\partial\sigma_{ij}|+ r^2|\partial^2\sigma_{ij}|=O(r^{-\tau}),$$
 where $r$ and $\partial$ denote the Euclidean distance and the standard derivative operator on $\mathbb{R}^n$ with the standard metric $\delta$, respectively.
The ADM mass is defined by
\begin{equation}\label{m1}
m_1(g):=m_{ADM}:=\frac 1{2(n-1)\omega_{n-1}}\lim_{r\to\infty}\int_{S_r}(g_{ij,i}-g_{ii,j})\nu_j dS,\end{equation}
where $\omega_{n-1}$ is the volume of $(n-1)$-dimensional standard unit sphere and  $S_r$ is the Euclidean coordinate sphere, $dS$ is the volume element on $S_r$ induced by the Euclidean metric, $\nu$ is the outward unit normal vector to $S_r$ in $\mathbb{R}^n$ and $g_{ij,k}= \partial_kg_{ij}$
are  the ordinary partial derivatives.

In a seminal paper Bartnik  \cite{Bar} proved that the ADM mass is well-defined for asymptotically flat Riemannian manifolds with a suitable decay  order $\tau$ and it is a geometric invariant. Precisely, it does not depend on the choice of the coordinates, provided
\begin{equation}\label{tau1}\tau>\frac {n-2}2.\end{equation}
  With this restriction, the ADM mass cannot be defined for many other interesting  asymptotically flat Riemannian manifolds. For example, the
following metric
\begin{equation}\label{metric}
g_{\rm Sch}^{(2)}=\bigg(1-\frac{2m}{\rho^{\frac n2-2}}\bigg)^{-1} d\rho^2+\rho^2d\Theta^2=\bigg(1+\frac {{m}}{2r^{\frac{n-4}2}}\bigg)^{\frac {8}{n-4} }g_{{\mathbb R}^n},
\end{equation}
plays an important role as the Schwarzschild metric in the (pure) Gauss-Bonnet gravity \cite{CTZ}.
Its decay order is $\frac {n-4} 2$, which is smaller than $\frac{n-2} 2$. Here $d\Theta^2$ is the standard metric on ${\mathbb S}^{n-1}$.
For the discussion of this metric and more general Schwarzschild type metrics, see Section $6$ below.

It is well-known that the ADM mass is very closely related to the scalar curvature. In fact, on an asymptotically
flat manifold with decay  order $\tau$, the scalar curvature has the following expression \cite{Schoen1}
\[R(g)=\partial_j(g_{ij,i}-g_{ii,j})+O(r^{-2\tau-2}).\]
  From this expression one can check that
\[\lim_{r\to \infty}\int_{S_r}(g_{ij,i}-g_{ii,j})\nu_j dS,\]
is well defined, provide that $\tau >\frac {n-2}2$
and $R$ is integrable. This term gives the ADM mass after a normalization. From this interpretation one can easily see the mathematical meaning of the ADM mass.
This also motivates us to introduce a new mass by using the following  second Gauss-Bonnet curvature\footnote{The second named author would like to thank Professor Schoen for his suggestion
to use this way to define a mass by using the $\sigma_k$-curvature ($k\ge 2$) in Toronto in 2005.}
\[L_2
=\|Rm\|^2-4\|Ric\|^2+R^2= \|W\|^2+8(n-2)(n-3)\s_2,\]
where $Rm$, $Ric$ and $W$ denote  the Riemannian curvature tensor, Ricci tensor and the Weyl tensor respectively, and $\s_2$ is the so-called $\s_2$-scalar curvature. More discussion about the Gauss-Bonnet curvature and the $\s_2$-scalar curvature
will be given in the next section.
Throughout the paper we use the Einstein summation convention.

\begin{defi}[Gauss-Bonnet-Chern Mass]\label{m2} Let $n\ge 5$.
Suppose that $(\cM^n,g)$ is an asymptotically flat manifold of decay order
\begin{equation}\label{eq0.1}
\tau>\frac{n-4}{3},
\end{equation}
 and the second Gauss-Bonnet curvature $L_2$ is integrable on $(\cM^n,g)$. We define the Gauss-Bonnet-Chern mass by
\begin{equation}\label{mass}
m_2(g):=m_{GBC}(g)=c_2(n)\lim_{r\rightarrow\infty}\int_{S_r}P^{ijkl}\partial_l g_{jk} \nu_{i}dS,
\end{equation}
where
\begin{equation}\label{constant}
c_2(n)=\frac 1 {2(n-1)(n-2)(n-3)\omega_{n-1}},
\end{equation} $\nu$ is the outward unit normal vector to $S_r$, $dS$ is the area element of $S_r$ and the tensor $P$ is defined by
$$P^{ijkl}=R^{ijkl}+R^{jk}g^{il}-R^{jl}g^{ik}-R^{ik}g^{jl}+R^{il}g^{jk}+\frac12R(g^{ik}g^{jl}-g^{il}g^{jk}).$$
\end{defi}
We remark that when $n=4$ one can also define the $m_2$ mass, but in this case (i.e., $n=4$) $m_2$ always vanishes. See also the discussion in Section 8.
In fact one can easily check that $m_2$ vanishes for asymptotically flat manifolds of decay order larger than $\frac{n-4}2$.
 Hence the ordinary Schwarzschild
metric
 $$g_{\rm Sch}^{(1)}=\bigg(1-\frac{2m}{\rho^{n-2}}\bigg)^{-1} d\rho^2+\rho^2d\Theta^2=\bigg(1+\frac {m}{2r^{n-2}}\bigg)^{\frac {4}{n-2} }g_{{\mathbb R}^n},$$
 considered in the Einstein gravity has a vanishing GBC mass whenever it can be defined, for it has the decay order $\tau=n-2>\frac{n-4}2$.
For the metric given in (\ref{metric}) one can check that the Gauss-Bonnet mass $m_2(g)=m^2$, which is nonnegative. See Section $6$ below.

Our work is partly motivated by the study of the $\sigma_2$-curvature and partly by the study of Einstein-Gauss-Bonnet gravity,
in which there is a similar mass defined for the Gauss-Bonnet gravity
\begin{equation}\label{eq0.2} R+\Lambda+\a L_2,\end{equation}
where $\Lambda$ is the cosmological constant and $\a$ is a parameter.
In contrast, if one considers only the term $L_2$, it is called the pure Gauss-Bonnet, or pure Lovelock gravity in physics. The study of Einstein-Gauss-Bonnet gravity
was initiated by the work of Boulware, Deser, Wheeler \cite{BD}, \cite{Wh}.
A mass for  (\ref{eq0.2}) was introduced by  Deser-Tekin \cite{DT1} and \cite{DT2}. See also \cite{DM,Pa,CO} and especially \cite{CTZ} and references therein.

\

Similar to  the work of Bartnik for the ADM mass, we first show that the GBC mass $m_2$ is a geometric invariant.

\begin{theo}\label{mainthm}
Suppose that $(\cM^n,g)$ $(n\geq5)$ is an asymptotically flat manifold of decay order
$
\tau>\frac{n-4}{3}$
 and $L_2$ is integrable on $(\cM^n,g)$. Then the Gauss-Bonnet-Chern mass $m_2$ is well-defined and does not depend on the choice of the  coordinates used in the definition.
\end{theo}

Now it is  natural to ask:

\

{\it Is the  GBC mass $m_2$ nonegative when the Gauss-Bonnet curvature $L_2$ is nonnegative?}

\

Due to the lack of methods, we can not yet attack this question for a general
asymptotically flat manifold. Instead, we leave this question as a
conjecture and provide a strong support in the following
result. Precisely, the problem has an affirmative answer,
if the asymptotically flat manifold $\cM^n$ can be embedded in ${\mathbb R}^{n+1}$ as a graph
over ${\mathbb R}^n$.
\begin{defi}\label{f}
	Let $f:\mathbb R^n \to \mathbb R$ be a smooth function and let $f_i$, $f_{ij}$ and $f_{ijk}$ denote
the first, the second and the third derivatives of $f$ respectively.  $f$ is called {\rm an asymptotically flat function of order
	 $\tau$} if
\begin{align*}
	f_i(x)&= O(|x|^{-\tau/2}), \\
	|x||f_{ij}(x)|+|x|^2|f_{ijk}(x)| &= O(|x|^{-\tau/2})
\end{align*}
at infinity for some $\tau > (n-4)/3$.
\end{defi}

\begin{theo}[Positive Mass Theorem]\label{mainthm2}{
Let $(\cM^n,g)=(\mathbb{R}^n,\delta+df\otimes df)$ be the graph of a smooth asymptotically flat function $f:\mathbb{R}^n\rightarrow \mathbb{R}$ in Definition \ref{f}. Assume  $L_2$ is integrable on $(\cM^n,g)$,
then
$$m_2=\frac {c_2(n)}{2}\int_{\cM^n}\frac{L_2}{\sqrt{1+|\nabla f|^2}}dV_g,$$
where $c_2(n)$ is the normalization constant defined in (\ref{constant}), the gradient and norm in $|\nabla f|$ are all with respect to the Euclidean metric $\delta$.
In particular, $L_2\geq 0$ yields $m_2\geq 0.$
}
\end{theo}

For  more details see Section $4$ below.
This result is motivated by the recent work of Lam \cite{Lam} mentioned above.
See also the work in \cite{HW}, \cite{dLG1} and \cite{dLG2}.
\begin{rema}
In this paper, for simplicity we focus on the mass defined by the second Gauss-Bonnet curvature.    From the proof one can easily generalize the results to the mass defined by the general Gauss-Bonnet curvature
\begin{eqnarray*}
L_k&=&\frac{1}{2^k}\d^{i_1i_2\cdots i_{2k-1}i_{2k}}
_{j_1j_2\cdots j_{2k-1}j_{2k}}{R_{i_1i_2}}^{j_1j_2}\cdots
{R_{i_{2k-1}i_{2k}}}^{j_{2k-1}j_{2k}}\\
&=&P_{(k)}^{ijlm}R_{ijlm}
\end{eqnarray*}
by (\ref{Lk}) and (\ref{Pk}) below with $k<n/2$. See  the discussion in Section 8.
\end{rema}
More interesting is that we have a Penrose type inequality, at least for  graphs.

\begin{theo}[Penrose Inequality] \label{mainthm3}
Let $\Omega$ be a bounded open set in $\mathbb{R}^n$ and $\Sigma=\partial\Omega.$ If $f:\mathbb{R}^n\setminus\Omega\rightarrow \mathbb{R}$ is a smooth
asymptotically flat function 
 such that each connected component of $\Sigma$ is in a level set of $f$ and $|\nabla f(x)|\rightarrow\infty$ as
$x\rightarrow\Sigma$. Let $\Sigma_i$ be the connected components of $\Sigma$  ($i=1,\cdots,l)$
and assume that each $\Sigma_i$ is
convex, then
\begin{eqnarray*}
m_2&\geq&\frac{c_2(n)}{2}\int_{\cM^n}\frac{L_2}{\sqrt{1+|\nabla f|^2}}d V_g
+\sum_{i=1}^l\frac{1}{4}\bigg(\frac{\int_{\Sigma_i} R}{(n-1)(n-2)\omega_{n-1}}\bigg)^{\frac{n-4}{n-3}}\\
&\geq&\frac{c_2(n)}{2}\int_{\cM^n}\frac{L_2}{\sqrt{1+|\nabla f|^2}}d V_g
+\sum_{i=1}^l\frac{1}{4}\bigg(\frac{\int_{\Sigma_i} H}{(n-1)\omega_{n-1}}\bigg)^{\frac{n-4}{n-2}}\\
&\geq&\frac{c_2(n)}{2}\int_{\cM^n}\frac{L_2}{\sqrt{1+|\nabla f|^2}}d V_g
+\sum_{i=1}^l\frac{1}{4}\bigg(\frac{|\Sigma_i|}{\omega_{n-1}}\bigg)^{\frac{n-4}{n-1}}\\
&\geq&\frac{c_2(n)}{2}\int_{\cM^n}\frac{L_2}{\sqrt{1+|\nabla f|^2}}d V_g
+\frac{1}{4}\bigg(\frac{|\Sigma|}{\omega_{n-1}}\bigg)^{\frac{n-4}{n-1}}.
\end{eqnarray*}
In particular,
$L_2\geq 0$ yields
$$m_2\geq \frac{1}{4}\bigg(\frac{\int_\Sigma R}{(n-1)(n-2)\omega_{n-1}}\bigg)^{\frac{n-4}{n-3}}\geq \frac{1}{4}\bigg(\frac{|\Sigma|}{\omega_{n-1}}\bigg)^{\frac{n-4}{n-1}}.$$ Moreover,  equalities are achieved by metric (\ref{metric}).
\end{theo}

The Penrose inequalities are optimal since equalities in the Penrose inequalities
are achieved  by  metric (\ref{metric}). Moreover, metric (\ref{metric}) can be realized as an induced metric of  a graph. See
Remark \ref{Schgraph} below.

This work, together with the work of Lam \cite{Lam}, gives a simple, but interesting application on the
positive mass theorem and the Penrose inequality for the ADM mass in Section 7 below.

Our results  open many interesting questions and establish a natural relationship between
many interesting functionals of intrinsic curvatures and extrinsic curvatures, which we will discuss at the end of this paper.

\

The rest of the paper is organized as follows. In Section 2 we recall the definitions of the Gauss-Bonnet curvature and the Lovelock curvature.
The new mass is defined  and proved  being a geometric invariant in Section 3. In Section 4 we show that the new mass is nonnegative
if the Gauss-Bonnet curvature is nonnegative for graphs.
The corresponding Penrose type inequality will be proved in Section 5. In Section 6, we will discuss metric (\ref{metric}), which is
 an important example and compute its Gauss-Bonnet-Chern mass explicitly. Applications to the ADM mass are given in Section 7 and  further generalizations,  problems and conjectures
are discussed  in Section 8.

\section{Lovelock curvatures}\label{sec5}

In this Section, let us  recall the work of Lovelock \cite{Lo} on generalized Einstein tensors.
Let
\[
 E=Ric -\frac 12 R g,
\]
be the Einstein tensor.
The Einstein tensor is very important in  physics, and certainly also in mathematics. It admits an important property,
namely it is a conserved quantity, i.e.,
\[\nabla_j E^{ij}=0,\]
where $\nabla$ is the covariant derivative with respect to the metric $g$.
And throughout the paper, we use the summation convention.
In \cite{Lo} Lovelock studied the classification of tensors $A$ satisfying
\begin{itemize}
 \item [(i)]$A^{ij}=A^{ji}$, i.e., $A$ is symmetric;
\item[(ii)] $A^{ij}=A^{ij}(g_{AB},g_{AB,C}, g_{AB,CD})$;
\item[(iii)] $\nabla_j A^{ij}=0$, i.e. $A$ is divergence-free;
\item [(iv)] $A^{ij}$ is linear in the second derivatives of
$g_{AB}$.
\end{itemize}
It is clear that the Einstein tensor satisfies all the conditions. In fact,  the Einstein tensor is
the unique tensor satisfying all four conditions, up to a  constant multiple.
Lovelock classified all the 2-tensors satisfying (i)--(iii).
 He proved that any 2-tensor satisfying (i)--(iii) has the following form
\[\sum_{j}\a_j E^{(j)},\]
with certain constants $\a_j$, $j\ge 0$, where
the 2-tensor $E^{(k)}$ is defined by
\[E^{(k)}_{ij}:=-\frac{1}{2^{k+1}}g_{l i}\d^{l  i_1i_2\cdots i_{2k-1}i_{2k}}
_{jj_1j_2\cdots j_{2k-1}i_{2k}}{R_{i_1i_2}}^{j_1j_2}\cdots
{R_{i_{2k-1}i_{2k}}}^{j_{2k-1}j_{2k}}.\]
Here
 the generalized Kronecker delta is defined by
\[
 \d^{j_1j_2 \dots j_r}_{i_1i_2 \dots i_r}=\det\left(
\begin{array}{cccc}
\d^{j_1}_{i_1} & \d^{j_2}_{i_1} &\cdots &  \d^{j_r}_{i_1}\\
\d^{j_1}_{i_2} & \d^{j_2}_{i_2} &\cdots &  \d^{j_r}_{i_2}\\
\vdots & \vdots & \vdots & \vdots \\
\d^{j_1}_{i_r} & \d^{j_2}_{i_r} &\cdots &  \d^{j_r}_{i_r}
\end{array}
\right).
\]
As a convention we set $E^{(0)}=1$. It is clear to see that $E^{(1)}$  is the Einstein tensor.
The tensor $E^{(k)}_{ij}$ is a very natural generalization of the Einstein tensor. We call $E^{(k)}$ the {\it $k$-th Lovelock curvature} and its
 trace up to a constant multiple
 \begin{equation}\label{Lk}
L_k:=\frac{1}{2^k}\d^{i_1i_2\cdots i_{2k-1}i_{2k}}
_{j_1j_2\cdots j_{2k-1}j_{2k}}{R_{i_1i_2}}^{j_1j_2}\cdots
{R_{i_{2k-1}i_{2k}}}^{j_{2k-1}j_{2k}},
\end{equation}
the {\it $k$-th Gauss-Bonnet curvature}, or simply the {\it Gauss-Bonnet curvature}. Both have been intensively studied in the Gauss-Bonnet gravity,
which is a generalization of the Einstein gravity. One could check  that $L_k=0$ if $2k>n$.
When $2k=n$, $L_k$ is in fact the Euler density, which was studied by Chern \cite{Chern1, Chern2} in his proof of the Gauss-Bonnet-Chern theorem. See also a nice survey \cite{Zhang}
on the  proof of the Gauss-Bonnet-Chern theorem.
 For
$k<n/2$, $L_k$ is therefore called the dimensional continued Euler density
in physics.
The above curvatures have been studied by many mathematicians and physicists, see for instance Pattersen \cite{Pat}
 and Labbi \cite{Labbi}.

In this paper we focus on the case $k=2$. The results can be generalized to  $k<n/2$. For the discussion see Section 8. One can also check that
\[
  E^{(2)}_{ij}=2RR_{ij}-4R_{is} R^s_{j} -4 R_{sl}
{{{R^s}_i}^l}_j+2R_{islt }{R_j}^{slt}-\frac12 g_{ij}L_2,
\]
and
\[
L_2=\frac 14 \delta
^{i_1i_2i_3i_4}_{j_1j_2j_3j_4} {R_{i_1i_2}}^{j_1j_2}{R_{i_3i_4}}^{j_3j_4}
=R_{ijsl}R^{ijsl}-4R_{ij}R^{ij}+R^2.
\]
$L_2$ is called the Gauss-Bonnet term in physics.
A direct computation gives the relation of $L_2$ with the $\sigma_2$-scalar curvature and the Weyl tensor as follows:
\begin{equation}\label{s1}\begin{array}{rcl}
 L_2 &=& \ds\vs \|W\|^2-4\frac{n-3}{n-2}\|Ric\|^2+\frac {n(n-3)}{(n-1)(n-2)} R^2\\
&=&\ds \vs\|W\|^2+\frac {n-3}{n-2}\bigg(\frac{n}{n-1}R^2-4\|Ric\|^2\bigg)\\
&=&\ds \|W\|^2+8(n-2)(n-3)\s_2.
\end{array}
\end{equation}
Here the $\sigma_k$-scalar curvature $\s_2$ has been intensively studied in the $\sigma_k$-Yamabe problem, which is first
studied by Viaclovsky \cite{Via2} and Chang-Gursky-Yang \cite{CGY}. For the study of the $\sigma_k$-Yamabe problem, see  for example the surveys
of Guan \cite{Guan} and Viaclovsky \cite{Via1}.

As a generalization of the Einstein metric, the solution of the following equation
is called {\it (string-inspired) Einstein-Gauss-Bonnet} metric
\[
 E^{(2)}_{ij}=\lambda g_{ij}.
\]
$E^{(2)}$ was already given by Lanczos \cite{Lan} in 1938 and is called the Lanczos tensor.
If $g$ is such a metric, it is obvious that
\[
 \lambda= \frac 1 n g^{ij} E^{(2)}_{ij}=\frac{4-n}{2n} L_2=\frac{4-n}{2n}\bigg( 8(n-2)(n-3)\s_2(g)+\|W\|^2\bigg).
\]
Since $E^{(2)}$ is  divergence-free,
$\lambda$ must be a constant in this case. This is a Schur type result. An almost Schur lemma for $E^{(k)}$
was proved in \cite{GWX}, which generalizes
a result of Andrews, De Lellis-Topping \cite{DT}.

\section{The Gauss-Bonnet-Chern  Mass}
In this section, we will introduce a new mass by using the Gauss-Bonnet curvature for asymptotically flat manifolds. This would be compared with the ADM mass which can be defined from the scalar curvature as indicated in the introduction. Moreover, following the approach from \cite{Bar} (see also \cite{LP}), we are able to show that this new  mass is a geometric invariant, i.e. it does not depend on the choice of  coordinates at infinity.

Recall that the Gauss-Bonnet curvature is defined by
$$L_2=R_{ijkl}R^{ijkl}-4R_{ij}R^{ij}+R^2.$$
One  crucial key to define a new mass is the observation that the Gauss-Bonnet curvature has the following decomposition
$$L_2=R_{ijkl}P^{ijkl},$$
where
\begin{equation}\label{P}
P^{ijkl}=R^{ijkl}+R^{jk}g^{il}-R^{jl}g^{ik}-R^{ik}g^{jl}+R^{il}g^{jk}+\frac12R(g^{ik}g^{jl}-g^{il}g^{jk}).
\end{equation}
This decomposition of $L_2$ will play a crucial role in the following discussion.
It is very easy to see that this $(0,4)$ tensor $P$ has the same symmetric property as the Riemannian curvature tensor, namely,
\begin{equation}\label{antisymmetry}
P^{ijkl}=-P^{jikl}=-P^{ijlk}=P^{klij}.
\end{equation}
Also one can easily check that $P$ satisfies the first Bianchi identity.
Another key ingredient is  that $P$ is divergence-free.  Before we discuss further, let us clarify the convention for the Riemannian curvature first:
$$
R_{ijkl}=R_{ijl}^m g_{mk},\quad
R_{ik}=g^{jl}R_{ijkl}=R^j_{jik}.
$$
\begin{lemm}\label{divergence-free}
$$
\nabla_iP^{ijkl}=\nabla_jP^{ijkl}=\nabla_kP^{ijkl}=\nabla_lP^{ijkl}=0.
$$
\end{lemm}
\begin{proof}
This lemma follows directly from the differential Bianchi identity.
\begin{eqnarray*}
\nabla_iP^{ijkl}&=&-\nabla^k R^{ijl}_{\quad i}-\nabla^l {{R^{ij}}_i}^k+\nabla^l R^{jk}-\nabla^k R^{jl}\\
&&-\frac12 \nabla^kR g^{jl}+\frac12\nabla^lR g^{jk}+\frac12 \nabla_i R(g^{ik}g^{jl}-g^{il}g^{jk})\\
&=&\nabla^k R^{jl}-\nabla^l R^{jk}+\nabla^l R^{jk}-\nabla^k R^{jl}-\frac12 \nabla^kR g^{jl}+\frac12\nabla^lR g^{jk}\\
&&+\frac12 \nabla^kR g^{jl}-\frac12\nabla^lR g^{jk}\\
&=&0.
\end{eqnarray*}
The rest follows  from the symmetry property (\ref{antisymmetry}) of $P$.
\end{proof}

This divergence-free property of $P$ was observed already in physics literature, see for instance \cite{Davis}.
In view of Lemma \ref{divergence-free}, we are able to
derive the corresponding expression of the mass-energy in the Einstein Gauss-Bonnet gravity.
We observe that for the asymptotically flat manifolds, the Gauss-Bonnet
curvature can be expressed as a divergence term besides some terms with
faster decay. First, in the local coordinates, the curvature tensor is expressed as
$$R_{ijk}^{m}=\partial_i \Gamma^m_{jk}-\partial_j \Gamma^m_{ik}+\Gamma^m_{is}\Gamma^s_{jk}-\Gamma^m_{js}\Gamma^s_{ik}.$$
 From  the divergence-free property of $P$
and the fact that the
quadratic terms of Christoffel symbols have faster decay, we compute
\begin{eqnarray}\label{divergenceL2}
L_2&=&R_{ijkl}P^{ijkl}=g_{km}R^m_{ijl}P^{ijkl}\nonumber\\
&=&g_{km}(\partial_i \Gamma^m_{jl}-\partial_j \Gamma^m_{il})P^{ijkl}+O(r^{-4-3\tau})\nonumber\\
&=&g_{km}\bigg[\nabla_i(\Gamma^m_{jl}P^{ijkl})-\nabla_j(\Gamma^m_{il}P^{ijkl})\bigg]+O(r^{-4-3\tau})\nonumber\\
&=&\frac 12\nabla_i\bigg[(g_{jk,l}+g_{kl,j}-g_{jl,k})P^{ijkl}\bigg]-\frac 12\nabla_j\bigg[(g_{ik,l}+g_{kl,i}-g_{il,k})P^{ijkl}\bigg]+O(r^{-4-3\tau})\nonumber\\
&=&2\nabla_i\bigg(g_{jk,l}P^{ijkl}\bigg)+O(r^{-4-3\tau})\nonumber\\
&=&2\partial_i\bigg(g_{jk,l}P^{ijkl}\bigg)+O(r^{-4-3\tau}),
\end{eqnarray}
where the fifth equality follows from (\ref{antisymmetry}).\\

With this divergence expression \eqref{divergenceL2} of $L_2$, one can check that the limit
$$\lim_{r\rightarrow\infty}\int_{S_r}P^{ijkl}\partial_l g_{jk}\nu_{i}dS,$$
exists and is finite provided that $L_2$ is integrable and the decay order $\tau>\frac{n-4}{3}$, and hence we have:
\begin{theo}\label{thm1}
Suppose that $(\cM^n,g)(n\geq5)$ is an asymptotically flat manifold of decay order $\tau>\frac{n-4}{3}$ and the  Gauss-Bonnet curvature $L_2$ is integrable on $(\cM^n,g)$, then the mass $m_2(g)$ defined  in Definition \ref{m2} is well-defined.
\end{theo}
We call $m_2$ the Gauss-Bonnet-Chern mass, or just the GBC mass.
The definition of the Gauss-Bonnet-Chern mass involves the choice of  coordinates at infinity. So it is natural to ask if it is a geometric invariant which does not depend on the choice of coordinates as the ADM mass. We have an affirmative answer.

\begin{theo}
If $(\cM^n,g)$ is asymptotically flat of order $\tau>\frac{n-4}{3}$ and $L_2$ is integrable, then $m_2(g)$ depends only on the metric $g$.
\end{theo}
\begin{proof}
The argument follows closely the one given by Bartnik  in the proof of ADM mass \cite{Bar}. See also \cite{LP, M}. The key is to realize that when one changes the coordinates, some extra terms which do not decay fast enough to have vanishing integral at infinity can be gathered in a divergence form.  The first step is the same as observed in \cite{Bar,LP}. For the convenience of the reader, we sketch it.\\
\vspace{1mm}

 \noindent{\bf Step 1.}
  Suppose $\{x^i\}$ and $\{\hat x^i\}$ are two choices of
coordinates at infinity on $\cM\smallsetminus K.$  In view of \cite{Bar}, after composing with an Euclidean motion, we may assume
$$\hat x^i=x^i+\varphi^i,\;\quad \mbox{where}\;\,\varphi^i\in C^{2,\alpha}_{1-\tau},$$ for some $0<\alpha<1.$ For the definition of these weighted spaces, please refer to \cite{Bar,LP} for more details.
Then the radial distance functions $r=|x|$ and $\hat r=|\hat x|$ are related by
$$C^{-1}r\leq \hat r\leq Cr, \quad\mbox{with some constant}\quad C>1.$$
Let  $S_R=\{x:r=R\}$ and  $\hat S_R=\{\hat x:\hat r=R\}$ be two spheres and $A_R=\{\hat x:C^{-1}R\leq \hat r\leq CR\}$ an annulus.
 The divergence theorem yields
\begin{eqnarray*}
\bigg|\int_{\hat S_R}P^{ijkl}\partial_l g_{jk}\hat \nu_i d\hat S-\int_{S_R}P^{ijkl}\partial_l g_{jk} \nu_i d S\bigg|
&\leq&\int_{A_R}\big|\partial_i(P^{ijkl}\partial_l g_{jk})\big|dx.
\end{eqnarray*}
Due to (\ref{divergenceL2})
together with the assumption that $L_2$ is integrable and $\tau>\frac{n-4}{3}$, the integral
$$\int_{A_R}\big|\partial_i(P^{ijkl}\partial_l g_{jk})\big|dx\rightarrow0\quad \mbox{as}\quad R\rightarrow\infty.$$
Therefore we can replace $S_{\hat r}$ by $ S_r$ in the definition of $m_2(g)$ without changing the mass.\\
\vspace{1mm}

\noindent {\bf Step 2.} Denote $\partial_i=\frac{\partial}{\partial x^i}$, $\hat{\partial}_i=\frac{\partial}{\partial \hat{x}^i}$, $g_{ij}=g(\partial_i,\partial_j)$ and $\hat g_{ij}=g(\hat{\partial}_i,\hat{\partial}_j),$ then we have
\begin{eqnarray}\label{diffenceofg}
\hat{\partial}_i&=&\partial_i-\partial_i\varphi^s\partial_s+O(r^{-\tau}),\nonumber\\
\hat g_{ij}&=&g_{ij}-\partial_i\varphi^j-\partial_j\varphi^i+O(r^{-2\tau}),\nonumber\\
\hat{\partial}_k\hat{g}_{ij}&=&\partial_k g_{ij}-\partial_k\partial_i\varphi^j-\partial_k\partial_j\varphi^i+O(r^{-1-2\tau}).
\end{eqnarray}
We compute
\begin{eqnarray*}
&&\int_{S_r}\bigg(P^{ijkl}\partial_l g_{jk}-\hat{P}^{ijkl}\hat{\partial}_l\hat{g}_{jk}\bigg)\nu_i d S\\
&=&\int_{S_r}P^{ijkl}(\partial_l g_{jk}-\hat{\partial}_l\hat{g}_{jk})\nu_i d S+\int_{S_r}(P^{ijkl}-\hat{P}^{ijkl})\hat{\partial}_l\hat{g}_{jk}\nu_i d S\\
&=&I+II.
\end{eqnarray*}
In view of  Lemma \ref{divergence-free}, together with (\ref{antisymmetry}) and (\ref{diffenceofg}), we compute
\begin{eqnarray*}
I&=&\int_{S_r}P^{ijkl}(\partial_l\partial_j\varphi^k+\partial_l\partial_k\varphi^j)\nu_i d S+\int_{S_r}O(r^{-3-3\tau})d S\\
&=&\int_{S_r}P^{ijkl}(\partial_l\partial_j\varphi^k)\nu_i d S+\int_{S_r}O(r^{-3-3\tau})d S\\
&=&\int_{S_r}P^{ijkl}(\partial_j\partial_l\varphi^k)\nu_i d S+\int_{S_r}O(r^{-3-3\tau})d S\\
&=&\int_{S_r}P^{ijkl}\partial_j\big((\partial_l\varphi^k) \nu_i\big)d S+\int_{S_r}O(r^{-3-3\tau})d S\\
&=&\int_{S_r}\partial_j\big(P^{ijkl}(\partial_l\varphi^k) \nu_i\big) d S+\int_{S_r}O(r^{-3-3\tau})d S\\
&=&\int_{S_r}[\partial_j-\langle \nu, \partial_j\rangle \nu]\big(P^{ijkl}(\partial_l\varphi^k) \nu_i\big) d S+\int_{S_r}\langle \nu, \partial_j\rangle  \nu\big(P^{ijkl}(\partial_l\varphi^k) \nu_i\big) d S+\int_{S_r}O(r^{-3-3\tau})d S,
\end{eqnarray*}
where the fourth equality follows from (\ref{antisymmetry}) and $\nu(f)=\frac{\partial}{\partial \nu} f$.
The first integral in $I$ vanishes from the divergence theorem. We will
show that the second integral  vanishes.

Since on the coordinate sphere $S_r$, the outward unit normal vector induced by the Euclidean metric $\nu{\triangleq}\nu^i\frac{\partial}{\partial x^i}=\nabla r,$
we thus  have $$ \nu_i\triangleq \delta_{ij}\nu^j=\nu^i=\frac{x^i}{r}.$$
By (\ref{antisymmetry}), we derive
\begin{eqnarray*}
&&\int_{S_r}\langle \nu, \partial_j\rangle \nu\big(P^{ijkl}(\partial_l\varphi^k)\nu_i\big) d S\\
&=&\int_{S_r}\nu_j\nu^t\frac{\partial}{\partial x^t}\big(P^{ijkl}(\partial_l\varphi^k)\nu_i\big) d S\\
&=&\int_{S_r}\nu_j \nu^t \partial_t\big(P^{ijkl}\partial_l\varphi^k\big)\nu_i d S+ \int_{S_r}\nu_j \nu^t P^{ijkl}(\partial_l\varphi^k)\frac{\partial}{\partial x^t}(\nu_i) d S\\
&=&\int_{S_r}\frac{x^jx^t}{r^2} P^{ijkl}(\partial_l\varphi^k)(\frac{\delta_{it}}{r}-\frac{x^ix^t}{r^3}) d S\\
&=&0.
\end{eqnarray*}
Hence we obtain
$$I=\int_{S_r}O(r^{-3-3\tau}) d S.
$$
For the second term $II$, we calculate
$$II=\int_{S_r}\bigg[(P^{ijkl}-\hat{P}^{ijkl})\partial_l g_{jk}\nu_i+O(r^{-3-3\tau})\bigg]dS,$$
and
\begin{eqnarray*}
P^{ijkl}-\hat{P}^{ijkl}&=&P_{ijkl}-\hat P_{ijkl}+O(r^{-2-2\tau})\\
&=&(R_{ijl}^k-\hat R_{ijl}^k)+(R_{jk}-\hat R_{jk})g_{il}-(R_{jl}-\hat R_{jl})g_{ik}-(R_{ik}-\hat R_{ik})g_{jl}\\
&&+(R_{il}-\hat R_{il})g_{jk}+\frac12(R-\hat R)(g_{ik}g_{jl}-g_{il}g_{jk})+O(r^{-2-2\tau}).
\end{eqnarray*}
From the expression of the curvature tensor and the Ricci tensor in local coordinates
\begin{eqnarray*}
R_{ijl}^k&=&-\frac12(\partial_i\partial_k g_{jl}-\partial_i\partial_l
g_{jk}-\partial_j\partial_k g_{il}+\partial_j\partial_l g_{ik})+O(r^{-2-2\tau}),\\
R_{jk}&=&\frac12(\partial_i\partial_k g_{ji}-\partial_i\partial_i g_{jk}
-\partial_j\partial_k g_{ii}+\partial_j\partial_i g_{ik})+O(r^{-2-2\tau}),
\end{eqnarray*}
and  the difference
$$\hat{\partial}_k\hat{g}_{ij}=\partial_k g_{ij}-\partial_k
\partial_i\varphi^j-\partial_k\partial_j\varphi^i+O(r^{-1-2\tau}),$$
we have
\begin{eqnarray*}
R_{ijl}^k-\hat R_{ijl}^k&=&-\frac12\big[\partial_i\partial_k\partial_j\varphi^l+\partial_i\partial_k\partial_l\varphi^j
-\partial_i\partial_l\partial_j\varphi^k-\partial_i\partial_l\partial_k\varphi^j-\partial_j\partial_k\partial_i\varphi^l\\
&&-\partial_j\partial_k\partial_l\varphi^i
+\partial_j\partial_l\partial_i\varphi^k+\partial_j\partial_l\partial_k\varphi^i\big]+O(r^{-2-2\tau})\\
&=&O(r^{-2-2\tau}).
\end{eqnarray*}
Similarly, we have
$$R_{jk}-\hat R_{jk}=O(r^{-2-2\tau}).$$
Thus we obtain $$II=\int_{S_r}O(r^{-3-3\tau}) d S.$$
Combining the two things together, we obtain
$$\int_{S_r}(P^{ijkl}\partial_l g_{jk}-\hat{P}^{ijkl}\hat{\partial}_l
\hat{g}_{jk})\nu_i d S=I+II=\int_{S_r}O(r^{-3-3\tau}) d S,$$
which implies  that
$$\lim_{r\rightarrow\infty}\int_{S_r}(P^{ijkl}\partial_l g_{jk}-
\hat{P}^{ijkl}\hat{\partial}_l\hat{g}_{jk})\nu_i d S=0, $$
when $\tau>\frac{n-4}{3}$.
Therefore we conclude $m_2(g)=m_2(\hat g)$ and  finish the proof.
\end{proof}

For the Euclidean metric, the GBC mass $m_2$ is trivially equal to zero. Examples with non-vanishing GBC mass  will be given in Section $6$ later.

\

\section{Positive mass theorem in the graph case}
In this section, we investigate the special case that asymptotically flat
manifolds are given as graphs of asymptotically flat functions over Euclidean space $\mathbb{R}^n.$
As in the Riemannian positive mass theorem studied by  Lam \cite{Lam},
for the  Gauss-Bonnet-Chern mass, we can show that the corresponding Riemannian positive mass theorem holds
for graphs when the Gauss-Bonnet curvature replaces  the scalar curvature in all dimensions $n\geq 5$.

Following the notation in \cite{Lam}, let $(\cM^n,g)=(\mathbb{R}^n,\delta+df\otimes df)$
be the graph of a smooth asymptotically flat function $f:\mathbb{R}^n\rightarrow \mathbb{R}$
 defined as in Definition \ref{f}. Then
$$g_{ij}=\delta_{ij}+f_if_j,$$
and the inverse of $g_{ij}$ is
\begin{equation}\label{ginverse}
g^{ij}=\delta_{ij}-\frac{f_if_j}{1+|\nabla f|^2},
\end{equation}
where the norm and the derivative  $\nabla f$ are taken with respect to
the flat metric $\delta.$ It is clear that such a graph is an asymptotically flat manifold
of order $\tau$ in the sense of Definition \ref{m2}.
The Christoffel symbols $\Gamma_{ij}^k$ with respect to the metric $g$ and its derivatives can be computed directly
\begin{eqnarray}\label{Christoffel}
\Gamma_{ij}^k&=&\frac{f_{ij}f_k}{1+|\nabla f|^2},\\
\partial_l\Gamma_{ij}^k&=&\frac{f_{ijl}f_k+f_{ij}f_{kl}}{1+|\nabla f|^2}-\frac{2f_{ij}f_k f_sf_{sl}}{(1+|\nabla f|^2)^2}\nonumber.
\end{eqnarray}
The expression for the curvature tensor follows directly. For the convenience of  the reader, we include some computations   in the following lemma.
\begin{lemm}\label{curvature}
$$
R_{ijkl}=\frac{f_{ik}f_{jl}-f_{il}f_{jk}}{1+|\nabla f|^2}.
$$
\end{lemm}
 \begin{proof}
 We begin with the $(1,3)$-type curvature tensor:
 \begin{eqnarray*}
 R_{ijk}^l&=&\partial_i\Gamma_{jk}^l-\partial_j\Gamma_{ik}^l+\Gamma_{is}^l\Gamma_{jk}^s-\Gamma_{js}^l\Gamma_{ik}^s\\
 &=&\frac{f_{jki}f_l+f_{jk}f_{li}}{1+|\nabla f|^2}-\frac{2f_{jk}f_l f_{si}f_s}{(1+|\nabla f|^2)^2}-\frac{f_{ikj}f_l+f_{ik}f_{lj}}{1+|\nabla f|^2}\\
 &&+\frac{2f_{ik}f_l f_{sj}f_s}{(1+|\nabla f|^2)^2}+\frac{f_{is}f_{jk}f_l f_s-f_{js}f_{ik}f_l f_s}{(1+|\nabla f|^2)^2}\\
 &=&\frac{f_{li}f_{jk}-f_{ik}f_{lj}}{1+|\nabla f|^2}+\frac{(f_{ik}f_{sj}-f_{jk}f_{si})f_s f_l}{(1+|\nabla f|^2)^2}\\
 &=&\frac{f_{il}f_{jk}-f_{ik}f_{jl}}{1+|\nabla f|^2}+\frac{(f_{ik}f_{js}-f_{jk}f_{is})f_s f_l}{(1+|\nabla f|^2)^2}.
 \end{eqnarray*}
Then we have
\begin{eqnarray*}
R_{ijkl}&=&R_{ijl}^m g_{km}\\
&=&\left(\frac{f_{im}f_{jl}-f_{il}f_{jm}}{1+|\nabla f|^2}+\frac{(f_{il}f_{js}-f_{jl}f_{is})f_s f_m}{(1+|\nabla f|^2)^2}\right)(\delta_{mk}+f_mf_k)\\
&=&\frac{f_{ik}f_{jl}-f_{il}f_{jk}}{1+|\nabla f|^2}+\frac{(f_{im}f_{jl}-f_{il}f_{jm})f_m f_k}{1+|\nabla f|^2}+\frac{(f_{il}f_{js}-f_{jl}f_{is})f_s f_k}{1+|\nabla f|^2}\\
&=&\frac{f_{ik}f_{jl}-f_{il}f_{jk}}{1+|\nabla f|^2}.
\end{eqnarray*}
 \end{proof}
\begin{rema}
This proof uses the intrinsic definition of the curvature tensor. One can also calculate it from the Gauss formula via the extrinsic geometry.
\end{rema}
The divergence-free property of $P$ enables us to  express the Gauss-Bonnet curvature $L_2$ as a divergence term. This is a key ingredient to show the corresponding positive mass theorem for the GBC mass in the graph case.
\begin{lemm}\label{divergenceform}
$$
\partial_i(P^{ijkl}\partial_l g_{jk})=\frac 12 L_2.
$$
\end{lemm}
\begin{proof}
$$\partial_i(P^{ijkl}\partial_l g_{jk})=\partial_iP^{ijkl}\partial_l g_{jk}+P^{ijkl}\partial_i\partial_l g_{jk}.$$
We begin with the first term. Here it is important to use Lemma \ref{divergence-free} to eliminate the terms of the third order derivatives of $f$. In view of (\ref{antisymmetry}) and Lemma \ref{divergence-free}, we compute
\begin{eqnarray*}
&&(\partial_iP^{ijkl})\partial_l g_{jk}\\
&=&(\nabla_iP^{ijkl}-P^{sjkl}\Gamma_{is}^i-P^{iskl}\Gamma_{is}^j-P^{ijsl}\Gamma_{is}^k-P^{ijks}\Gamma_{is}^l)\partial_l g_{jk}\\
&=&\!-\left(P^{sjkl}\frac{f_{is}f_i}{1+|\nabla f|^2}+P^{iskl}\frac{f_{is}f_j}{1+|\nabla f|^2}+P^{ijsl}\frac{f_{is}f_k}{1+|\nabla f|^2}+P^{ijks}\frac{f_{is}f_l}{1+|\nabla f|^2}\right)(f_{jl}f_k+f_{kl}f_j)\\
&=&-\left(P^{sjkl}\frac{f_{is}f_{jl}f_if_k}{1+|\nabla f|^2}+P^{ijsl}\frac{f_{is} f_{jl}|\nabla f|^2+f_{is} f_{kl} f_k f_j}{1+|\nabla f|^2}+P^{ijks}\frac{f_{is} f_{jl}f_k f_l+f_{is} f_{kl} f_j f_l}{1+|\nabla f|^2}\right)\\
&=&-\frac{P^{ijkl}}{1+|\nabla f|^2}\left((f_{is} f_{jl}+f_{il}f_{js})f_k f_s+(f_{ik} f_{sl}+f_{il}f_{ks})f_j f_s+|\nabla f|^2f_{ik} f_{jl}\right)\\
&=&-\frac{|\nabla f|^2}{1+|\nabla f|^2}f_{ik}f_{jl}P^{ijkl},
\end{eqnarray*}
where we have relabeled the indices in the fourth equality and
used property (\ref{antisymmetry}) in the third and fifth equalities.

The second term is also simplified by (\ref{antisymmetry}). Namely,
\begin{eqnarray*}
P^{ijkl}\partial_i\partial_l g_{jk}&=&P^{ijkl}\partial_i\partial_l(f_j f_k)\\
&=&P^{ijkl}(f_{jli}f_k+f_{kli}f_j+f_{ki}f_{jl}+f_{ji}f_{kl})\\
&=&P^{ijkl}f_{ki}f_{jl}.
\end{eqnarray*}
Combining these two terms together, we arrive at
\begin{eqnarray}\label{addx1}
\partial_i(P^{ijkl}\partial_l g_{jk})
&=&P^{ijkl}\frac{f_{ik}f_{jl}}{1+|\nabla f|^2}\nonumber\\
&=&\frac 12P^{ijkl}\big(\frac{f_{ik}f_{jl}-f_{il}f_{jk}}{1+|\nabla f|^2}\big).
\end{eqnarray}
Recall that $$L_2=P^{ijkl}R_{ijkl},$$
and invoke the expression of the curvature tensor in Lemma \ref{curvature}, we complete the proof of the lemma.
\end{proof}

 Lam showed
a similar result for the scalar curvature, which is the crucial observation in \cite{Lam}.  See also the first paragraph of Section
$8$ below.

\

Now we are ready to prove one of our main results, Theorem \ref{mainthm2}.

\

\noindent {\it Proof of  Theorem \ref{mainthm2}.}
In view of  Lemma \ref{divergenceform} and the divergence theorem in $(\mathbb{R}^n,\delta)$, we have
\begin{eqnarray*}
m_2&=&\lim_{r\rightarrow\infty}c_2(n)\int_{S_r}P^{ijkl}\partial_l g_{jk}\nu_{i}dS\\
&=&c_2(n)\int_{\mathbb{R}^n}\partial_i(P^{ijkl}\partial_l g_{jk})dV_{\delta}\\
&=&\frac {c_2(n)}{2}\int_{\mathbb{R}^n}L_2dV_{\delta}\\
&=&\frac  {c_2(n)}{2}\int_{{\cM^n}}\frac{L_2}{\sqrt{1+|\nabla f|^2}}dV_g,
\end{eqnarray*}
where the last equality holds due to the fact
$$dV_g=\sqrt{\mbox{det}g}dV_{\delta}=\sqrt{1+|\nabla f|^2}dV_{\delta}.$$
\qed

\section{Penrose inequality for graphs on $\mathbb{R}^n$ }

In this section we investigate the Penrose inequality related to the GBC mass for the manifolds which can be realized as graphs over $\mathbb{R}^n$.
Let $\Omega$ be a bounded open set in $\mathbb{R}^n$ and $\Sigma=\partial\Omega.$  If $f:\mathbb{R}^n\setminus\Omega\rightarrow \mathbb{R}$ is a smooth
asymptotically flat function such that each connected component of $\Sigma$ is in a level set of $f$ and
\begin{equation}\label{x1}
|\nabla f(x)|\rightarrow\infty \quad\hbox{  as }
x\rightarrow\Sigma,\end{equation}
then the graph of $f$, $(\cM^n,g)=(\mathbb{R}^n\setminus\Omega,\delta+df\otimes df),$ is an asymptotically flat manifold with an area  outer
minimizing
horizon $\Sigma.$  See Remark \ref{rem5.1} below. Without loss of generality we may assume that $\Sigma$  is included in $f^{-1}(0)$.
In this case one can
identify $\{(x, f(x)) \,|\, x\in \Sigma\}$ with $\Sigma$.

On $\Sigma$, the outer unit normal vector induced by $\delta$ is $$\nu\triangleq \nu^i\frac{\partial}{\partial x^i }=-\frac{\nabla
f}{|\nabla f|}.$$
Then
$$\nu^i=-\frac{f^i}{|\nabla f|}=-\frac{f_i}{|\nabla f|}\quad\mbox{and}\quad \nu_i\triangleq \delta_{ij}\nu^j=\nu^i.$$
As in the proof of Theorem \ref{mainthm2}, integrating by parts now gives an extra boundary term
\begin{eqnarray*}
m_2&=&\lim_{r\rightarrow\infty}c_2(n)\int_{S_r}P^{ijkl}\partial_l g_{jk} \nu_{i}dS\\
&=&\frac{c_2(n)}{2}\int_{\mathbb{R}^n\setminus\Omega}\frac{L_2}{\sqrt{1+|\nabla f|^2}}dV_g-c_2(n)\int_{\Sigma}P^{ijkl}\partial_l g_{jk}\nu_{i}d S\\
&=&\frac{c_2(n)}{2}\int_{\cM^n}\frac{L_2}{\sqrt{1+|\nabla f|^2}}dV_g-c_2(n)\int_{\Sigma}P^{ijkl}(f_{jl}f_k+f_{kl}f_j) \nu_{i}d S\\
&=&\frac{c_2(n)}{2}\int_{\cM^n}\frac{L_2}{\sqrt{1+|\nabla f|^2}}dV_g-c_2(n)\int_{\Sigma}P^{ijkl}f_{jl}f_k \nu_{i}d S,
\end{eqnarray*}
where the last term in the third equality vanishes because of the symmetry of $P$.
We derive from (\ref{P}) and (\ref{ginverse}) that
\begin{eqnarray*}
&&P^{ijkl}f_{jl}f_k \nu_{i}\\
&=&R^{ijkl}f_{jl}f_k \nu_{i}+R^{jk}(f_{ji}f_k -\frac{f_{jl}f_l f_k}{1+|\nabla f|^2}f_i)\nu_{i}-R^{ik}(f_{jj}f_k-\frac{f_{jl}f_l f_j}{1+|\nabla f|^2}f_k)
\nu_{i}\\
&&+R^{il}f_j\frac{f_{jl}}{1+|\nabla f|^2}\nu_{i}-R^{jl}f_{jl}\frac{f_i}{1+|\nabla f|^2}\nu_{i}+\frac12 R\frac{f_{jj}f_i-f_{ji}f_j}{1+|\nabla f|^2}
\nu_{i}\\
&=&I+II+III+IV+V+VI.
\end{eqnarray*}

Recall  that we have assumed that  $\Sigma$ is in a level set of $f$.
At any given point $p\in\Sigma$, we choose the  coordinates such that $\{\frac{\partial}{\partial{x^2}},\cdots,\frac{\partial}{\partial{x^n}}\}$ denotes the tangential space
 of $\Sigma$ and $\frac{\partial}{\partial{x^1}}$ denotes the normal direction of $\Sigma$. To clarify the notations, in the following
  we will use the convention that the Latin letters stand for the index: $1,2,\cdots, n$
and the Greek letters stand for the index: $2, \cdots, n$.  Now the computations are all done at the given point $p$.
It is easy to see that  $$f_{\alpha}=0 \quad\mbox{and}\quad f_{\alpha\beta}=A_{\alpha\beta}|\nabla f|=A_{\alpha\beta}|f_1|,$$
where $A_{\alpha\beta}$ is the second fundamental form of the isometric embedding $(\Sigma, h)$ into the  Euclidean space $\R^n$. In other words, $h$ is
the induced metric. Note that there is also an isometric embedding from $(\Sigma, h)$ into the graph as the boundary of the graph.

Before computing further, let us introduce a notation.  $H_k$ denotes the $k$-th mean curvature, which is defined by the $k$-th elementary symmetric function on the principal curvatures
of   the second fundamental form $A$.
In the following, we  calculate each  term in $P^{ijkl}f_{jl}f_k \nu_{i}$. The point in the computation is to distinguish the tangential direction and the normal direction carefully.
\begin{eqnarray*}
I&=&R^{1\alpha1\beta}f_{\alpha\beta}f_1\nu_1=R^{1\alpha1\beta}A_{\alpha\beta}|f_1|f_1(-\frac{f_1}{|f_1|})\\
&=&-R^{1\alpha1\beta}A_{\alpha\beta}f_1^2,\\
II&=&R^{j1}(f_{j1}f_1-\frac{f_{j1}f_1^3}{1+f_1^2}) \nu_1
=(R^{j1}f_{j1})\frac{f_1}{1+f_1^2}(-\frac{f_1}{|f_1|})\\
&=&-(R^{j1}f_{j1})\frac{f_1^2}{(1+f_1^2)|f_1|},\\
III&=&-R^{11}\left((H_1|f_1|+f_{11})f_1-\frac{f_{11}f_1^3}{1+f_1^2}\right) \nu_1\\
&=&-R^{11}\left(H_1|f_1|f_1+\frac{f_{11}f_1}{1+f_1^2}\right)(-\frac{f_1}{|f_1|})\\
&=&R^{11}\left(H_1 f_1^2+\frac{f_{11}f_1^2}{(1+f_1^2)|f_1|}\right),\\
IV&=&R^{1l}f_{1l}\frac{f_1}{1+f_1^2} \nu_1
=-R^{1l}f_{1l}\frac{f_1^2}{(1+f_1^2)|f_1|},\\
V&=&-(R^{jl}f_{jl})\frac{f_1}{1+f_1^2} (-\frac{f_1}{|f_1|})
=(R^{jl}f_{jl})\frac{f_1^2}{(1+f_1^2)|f_1|}\\
&=&(2R^{1l}f_{1l}+R^{\alpha\beta}f_{\alpha\beta}-R^{11}f_{11})\frac{f_1^2}{(1+f_1^2)|f_1|}\\
&=&(2R^{1l}f_{1l}+R^{\alpha\beta}A_{\alpha\beta}|f_1|-R^{11}f_{11})\frac{f_1^2}{(1+f_1^2)|f_1|},\\
VI&=&\frac 12 R\frac{(H_1|f_1|+f_{11})f_1-f_{11}f_1}{1+f_1^2}(-\frac{f_1}{|f_1|})\\
&=&-\frac 12 RH_1\frac{f_1^2}{1+f_1^2}.
\end{eqnarray*}
Noting that $II+IV$ cancels the first term of $V$ and the second term in $III$ cancels the third term in $V$ we get
\begin{eqnarray}\label{sum}
P^{ijkl}f_{jl}f_k \nu_{i}
&=&I+II+III+IV+V+VI\nonumber\\
&=&-R^{1\alpha1\beta}A_{\alpha\beta}f_1^2+R^{11}H_1 f_1^2+R^{\alpha\beta}A_{\alpha\beta}\frac{f_1^2}{1+f_1^2}-\frac 12 RH_1\frac{f_1^2}{1+f_1^2}.
\end{eqnarray}

Similarly, for the embedding  $(\Sigma^{n-1}, h)\hookrightarrow (\cM^n,g)$ we denote  the outer unit normal vector induced by $g$  by
$\tilde \nu$, and the corresponding second fundamental form by $\tilde A_{\alpha\beta}.$
Then a direct calculation gives
$$\tilde A_{\alpha\beta}=\frac{1}{\sqrt{1+f_1^2}}A_{\alpha\beta}.$$

\begin{rema} \label{rem5.1} From the above formula we have the following equivalent statements, provided $\Sigma\subset\mathbb{R}^n$ is strictly mean convex:
\begin{itemize} \item $|\nabla f|=\infty$ on $\Sigma$;
\item $\Sigma$ is minimal, i.e., $\tr \tilde A=0$;
\item $\Sigma$ is totally geodesic, i.e., $\tilde A=0$.
\end{itemize}
Therefore $\Sigma$ is an area-minimizing  horizon if and only if $|\nabla f|=\infty$ on $\Sigma$ and if and only if  $\Sigma$ is  totally geodesic. Hence
$|\nabla f|=\infty$  is a natural assumption.
\end{rema}

 By the Gauss equation, viewing $(\Sigma, h)$ as a hypersurface in $\R^n$, we have
$$\hat R^{\alpha\beta\gamma\delta}=A^{\alpha\gamma}A^{\beta\delta}-A^{\alpha\delta}A^{\beta\gamma},$$
where $\hat R^{\alpha\beta\gamma\delta}$ is the corresponding curvature tensor with respect to the induced metric $h$ on $\Sigma$.
On the other hand, viewing $(\Sigma, h)$ as a hypersurface in $(\cM, g)$ we have
\begin{eqnarray*}
\hat R^{\alpha\beta\gamma\delta}=R^{\alpha\beta\gamma\delta}+\tilde A^{\alpha\gamma}\tilde A^{\beta\delta}-\tilde A^{\alpha\delta}\tilde A^{\beta\gamma}=R^{\alpha\beta\gamma\delta}+\frac{A^{\alpha\gamma}A^{\beta\delta}-A^{\alpha\delta}A^{\beta\gamma}}{1+f_1^2},
\end{eqnarray*}
which yields
\begin{eqnarray}\label{Rie.tensor}
R^{\alpha\beta\gamma\delta}&=&\frac{f_1^2}{1+f_1^2}(A^{\alpha\gamma}A^{\beta\delta}-A^{\alpha\delta}A^{\beta\gamma})\nonumber\\
&=&\frac{f_1^2}{1+f_1^2}\hat R^{\alpha\beta\gamma\delta}.
\end{eqnarray}
Similarly, we have
\begin{eqnarray}\label{Ricci}
\tilde F^{\alpha\beta}&\triangleq& R^{\alpha\gamma\beta\delta}g_{\gamma\delta}\nonumber\\
&=&\frac{|\nabla f|^2}{1+|\nabla f|^2}(H_1A_{\alpha\beta}-A_{\alpha\gamma}A_{\gamma\beta})=\frac{|\nabla f|^2}{1+|\nabla f|^2} \hat R^{\alpha\beta}
\end{eqnarray}
and
\begin{equation}\label{Scalar}
\tilde F \triangleq R^{\alpha\beta}g_{\alpha\beta}=\frac{|\nabla f|^2}{1+|\nabla f|^2}(H_1^2-A_{\alpha\beta}A_{\alpha\beta})=\frac{|\nabla f|^2}{1+|\nabla f|^2} \hat R.
\end{equation}
We then go back to equality (\ref{sum}). Note the facts
$$R^{\alpha\beta}=\tilde F^{\alpha\beta}+R^{1\alpha1\beta} g_{11}=\tilde F^{\alpha\beta}+R^{1\alpha1\beta}(1+f_1^2),$$
and
$$R=\tilde F+2R^{11}g_{11}=\tilde F+2R^{11}(1+f_1^2).$$
We compute
\begin{eqnarray*}
P^{ijkl}f_{jl}f_k \nu_{i}&=&-R^{1\alpha1\beta}A_{\alpha\beta}f_1^2+R^{11}H_1 f_1^2+R^{\alpha\beta}A_{\alpha\beta}\frac{f_1^2}{(1+f_1^2)}-\frac 12 RH_1\frac{f_1^2}{1+f_1^2}\\
&=&-R^{1\alpha1\beta}A_{\alpha\beta}f_1^2+R^{11}H_1 f_1^2+\big[\tilde F^{\alpha\beta}+R^{1\alpha1\beta}(1+f_1^2)\big]A_{\alpha\beta}\frac{f_1^2}{(1+f_1^2)}\\
&&-\frac 12\big[\tilde F+2R^{11}(1+f_1^2)\big] H_1\frac{f_1^2}{1+f_1^2}\\
&=&(\tilde F^{\alpha\beta}A_{\alpha\beta}-\frac 12 \tilde FH_1) \frac{f_1^2}{1+f_1^2}\\
&=&(\hat R^{\alpha\beta}-\frac 12 \hat R h^{\alpha\beta})A_{\alpha\beta}\bigg(\frac{|\nabla f|^2}{1+|\nabla f|^2}\bigg)^2,
\end{eqnarray*}
where we have used (\ref{Ricci}) and (\ref{Scalar}).
Therefore we conclude
$$
\int_{\Sigma}P^{ijkl}f_{jl}f_k\nu_id S=\int_{\Sigma}\left(\frac{|\nabla f|^2}{1+|\nabla f|^2}\right)^2(\hat R^{\alpha\beta}-\frac 12 \hat R h^{\alpha\beta})A_{\alpha\beta}d S.
$$
One can check that
$$-(\hat R^{\alpha\beta}-\frac 12 \hat R h^{\alpha\beta})A_{\alpha\beta}=3H_3.$$

Making use of the assumption that $|\nabla f(x)|\rightarrow\infty$ as $x\rightarrow\Sigma$,
we have
\begin{eqnarray*}
m_2
&=&\frac{c_2(n)}{2}\int_{\cM^n}\frac{L_2}{\sqrt{1+|\nabla f|^2}}dV_g-c_2(n)\int_{\Sigma}P^{ijkl}f_{jl}f_k \nu_{i}d S\nonumber\\
&=&\frac{c_2(n)}{2}\int_{\cM^n}\frac{L_2}{\sqrt{1+|\nabla f|^2}}dV_g+c_2(n)\int_{\Sigma}\left(\frac{|\nabla f|^2}{1+|\nabla f|^2}\right)^2 3H_3 dS\nonumber\\
&=&\frac{c_2(n)}{2}\int_{\cM^n}\frac{L_2}{\sqrt{1+|\nabla f|^2}}dV_g+c_2(n)\int_{\Sigma} 3H_3 dS.
\end{eqnarray*}

To summarize, we have showed that
\begin{prop}\label{penrose}
Let $\Omega$ be a bounded open set in $\mathbb{R}^n$ and $\Sigma=\partial\Omega.$ If $f:\mathbb{R}^n\setminus\Omega\rightarrow \mathbb{R}$ is a smooth
asymptotically flat function such that each connected component of $\Sigma$ is in a level set of $f$ and $|\nabla f(x)|\rightarrow\infty$ as
$x\rightarrow\Sigma$.  Then
$$m_2=\frac{c_2(n)}{2}\int_{\cM^n}\frac{L_2}{\sqrt{1+|\nabla f|^2}}dV_g+c_2(n)\int_{\Sigma} 3H_3 dS.$$
\end{prop}
\vspace{5mm}
 Let $\Omega_i$ be connected components of $\Omega,i=1,\dots,k$ and
let $\Sigma_i=\partial\Omega_i$ and assume that each $\Omega_i$ is
convex.
The rest to show the Penrose inequality in the graph case is the same as the one in \cite{Lam}, that to use the Aleksandrov-Fenchel inequality \cite{AF} .
\begin{lemm}\label{Alexsandrov}
 Assume $\Sigma$ is a convex hypersurface in $\mathbb{R}^n$, then
$$
\begin{array}{llll}
\displaystyle \frac{1}{2(n-1)(n-2)(n-3)\omega_{n-1}}\int_{\Sigma} 3H_3 dS&\displaystyle\geq \frac{1}{4}\bigg(\frac{\int_\Sigma RdS}{(n-1)(n-2)\omega_{n-1}}\bigg)^{\frac{n-4}{n-3}}\\
&\displaystyle\geq \frac{1}{4}\bigg(\frac{\int_\Sigma HdS}{(n-1)\omega_{n-1}}\bigg)^{\frac{n-4}{n-2}}\\
&\displaystyle \geq \frac{1}{4}\bigg(\frac{|\Sigma|}{\omega_{n-1}}\bigg)^{\frac{n-4}{n-1}}.
\end{array}
$$
\end{lemm}

\begin{proof}
By  the Aleksandrov-Fenchel inequality, we infer that
 $$
 \begin{array}{llll}
\displaystyle\frac{1}{2(n-1)(n-2)(n-3)\omega_{n-1}}\int_{\Sigma} 3H_3 dS&\displaystyle\geq \frac{1}{4}\bigg(\frac{\int_\Sigma 2H_2dS}{(n-1)(n-2)\omega_{n-1}}\bigg)^{\frac{n-4}{n-3}}\\
&\displaystyle\geq \frac{1}{4}\bigg(\frac{\int_\Sigma H_1dS}{(n-1)\omega_{n-1}}\bigg)^{\frac{n-4}{n-2}}\\
&\displaystyle\geq \frac{1}{4}\bigg(\frac{|\Sigma|}{\omega_{n-1}}\bigg)^{\frac{n-4}{n-1}}.
\end{array}
$$
 On the other hand, it follows from the Gauss equation
 $2H_2=R$. Hence the desired result yields.
\end{proof}

Now we are ready to  finish  the proof of the  Penrose inequality.

\

\noindent {\it Proof of Theorem \ref{mainthm3}}.
In view of  Proposition \ref{penrose} together with  Lemma \ref{Alexsandrov},  we have showed the first part of Theorem \ref{mainthm3}.
It remains to check that  metric (\ref{metric}) (See also  Example \ref{Schwarzschild} below) attains the equality in the Penrose-type inequality. First, it follows from the calculation in Section 6 below
that the Gauss-Bonnet curvature $L_2$ with respect to  metric (\ref{metric}) is equal to $0$.
Its  horizon is $S_{\rho_0}=\{(\rho,\theta):\rho_0^{\frac n2-2}=2m\}$ which implies the right hand side of the Penrose-type inequality is
\begin{eqnarray*}
RHS&=&\frac14\bigg(\frac{\omega_{n-1}{\rho_0}^{n-1}}{\omega_{n-1}}\bigg)^{\frac{n-4}{n-1}}\\
&=&\frac 14{\rho_0}^{n-4}=\frac 14(2m)^2\\
&=&m^2,
\end{eqnarray*}
which is equal to the GBC mass $m_2$. See its computation in Section 6 below.
\qed

\begin {rema} By the work of Guan-Li \cite{GL} one can reduce the assumption of convexity of $\Sigma$ to the assumptions that $\Sigma$
is star-shaped, $H_1>0$, $R>0$ and $H_3$ is non-negative. See also the related work of \cite{CW}.

\end{rema}
\section{Generalized Schwarzschild metrics}
\begin{exam}\label{Schwarzschild}
$(\cM^n=I\times {\mathbb S}^{n-1},g)$ with coordinates $(\rho,\theta),$  general Schwarzschild metric is given by
$$g^k_{\rm Sch}=(1-\frac{2m}{\rho^{\frac nk-2}})^{-1}d\rho^2+\rho^2d\Theta^2,$$
where $d\Theta^2$ is the round metric on ${\mathbb S}^{n-1},$ $m\in \mathbb{R}$ is the `total mass' of corresponding black hole solution in the  Lovelock gravity \cite{Lo}.
When $k=1$ we recover the Schwarzschild solution of the Einstein gravity.\\
\end{exam}
Motivated by the Schwarzschild solution, the above metric can be also written as conformally flat one  which is more convenient for computation.
One can check that the corresponding  transformation
$$(1-\frac{2m}{{\rho}^{\frac nk-2}})^{-1}d\rho^2+\rho^2d\Theta^2=(1+\frac{m}{2r^{\frac nk-2}})^{\frac{4k}{n-2k}}(dr^2+r^2d\Theta^2).$$

For our purpose, in this paper we focus on  the case $k=2$, namely,
$$g^{(2)}_{\rm Sch}=(1-\frac{2m}{\rho^{\frac n2-2}})^{-1}d\rho^2+\rho^2d\Theta^2=(1+\frac{m}{2r^{\frac n2-2}})^{\frac{8}{n-4}}\delta,$$
where $\delta$ is the standard Euclidean metric, which was given in the introduction.

Next, we will study the correspondence between $m$ and the Gauss-Bonnet mass $m_2.$\\
Recall
\begin{eqnarray}\label{m2calculate}
m_2&=&\lim_{r\rightarrow\infty}\frac {1}{2(n-1)(n-2)(n-3)\omega_{n-1}}\int_{S_r}P^{ijkl}\partial_l g_{jk}\nu_idS.\nonumber
\end{eqnarray}
For the simplicity of notation, we introduce the notation of the Kulkarmi-Nomizu product denoted by $\odot$. More precisely,
\begin{eqnarray*}
&&(A\odot B)(X,Y,Z,W)\\
&:=&A(X,Z)B(Y,W)-A(Y,Z)B(X,W)-A(X,W)B(Y,Z)+A(Y,W)B(X,Z).
\end{eqnarray*}
Then we have the compressed expression
$$P=Rm-Ric\odot g+\frac 14 R(g\odot g).$$
Suppose $g=e^{-2u}\delta,$ a direct computation gives
\begin{eqnarray}\label{conformalcuv.}
Rm&=&e^{-2u}(\nabla_{\delta}^2u+du\otimes du-\frac 12|\nabla_{\delta}u|^2\delta)\odot \delta\nonumber\\
Ric&=&(n-2)\left(\nabla_{\delta}^2u+\frac{1}{n-2}(\Delta_{\delta}u)\delta+du\otimes du-|\nabla_{\delta}u|^2\delta \right)\\
R&=&e^{2u}\left(2(n-1)\Delta_{\delta}u-(n-1)(n-2)|\nabla_{\delta}u|^2\right).\nonumber
\end{eqnarray}
Then we have
\begin{eqnarray}\label{localP}
P&=& Rm-Ric\odot g+\frac 14 R(g\odot g)\nonumber\\
&=&Rm-e^{-2u}Ric\odot\delta+\frac 14e^{-4u}R(\delta\odot\delta)\nonumber\\
&=&(n-3)e^{-2u}\left(-\nabla_{\delta}^2 u+\frac 12(\Delta_{\delta}u)\delta-du\otimes du-\frac{n-4}{4}|\nabla_{\delta}u|^2\delta\right)\odot\delta.
\end{eqnarray}
which implies
\begin{eqnarray}\label{integralterm}
P^{ijkl}&=&(n\!-\!3)e^{6u}\big[-u_{ik}\delta_{jl}-u_{jl}\delta_{ik}+u_{il}\delta_{jk}+u_{jk}\delta_{il}+(u_{ss}\!-\!\frac{n-4}{2}u_s^2)(\delta_{ik}\delta_{jl}-\delta_{il}\delta_{jk})\nonumber\\
&&\qquad\qquad\quad-u_i u_k\delta_{jl}-u_j u_l\delta_{ik}+u_i u_l\delta_{jk}+u_j u_k\delta_{il}\big].
\end{eqnarray}
Since $g=e^{-2u}\delta,$ we have
$$
\int_{S_r}P^{ijkl}\partial_l g_{jk}\nu_idS
=\int_{S_r}P^{ijkl}(\partial_l e^{-2u})\delta_{jk}\nu_idS
=\int_{S_r}-2e^{-2u}P^{ijjl}u_l \nu_id S.
$$
 From (\ref{integralterm}), we have
\begin{eqnarray*}
&&-2e^{-2u}P^{ijjl}u_l\\
&=&-2(n-3)e^{4u}u_l[-u_{il}-u_{il}+nu_{il}+u_{ss}\delta_{il}+(u_{ss}-\frac{n-4}{2}u_s^2)(\delta_{il}-n\delta_{il})\\
&&\qquad\qquad\qquad\quad\;-u_i u_l-u_i u_l+nu_i u_l+u_s^2\delta_{il}]\\
&=&-2(n-2)(n-3)e^{4u}u_l[u_{il}-u_{ss}\delta_{il}+u_i u_l+\frac{n-3}{2}u_s^2\delta_{il}].
\end{eqnarray*}

We now consider a special case that $(\cM^n,g)$ is asymptotically flat  and conformally flat with a smooth spherically symmetric function, namely,
$g=e^{-2u(r)}\delta.$
Denoting the radial derivative of $u$ by $u_r\triangleq\frac{\partial u}{\partial r}$,  we have
\begin{eqnarray}\label{ui}
u_i&=&u_r\frac{x_i}{r},\\
u_{ij}&=&u_{rr}\frac{x_i x_j}{r^2}+u_r(\frac{\delta_{ij}}{r}-\frac{x_i x_j}{r^3}),\label{uij}
\end{eqnarray}
which yields
\begin{equation}\label{uii}
u_{ii}=u_{rr}+\frac{n-1}{r}u_r.
\end{equation}
It follows from (\ref{ui}), (\ref{uij}) and (\ref{uii}) that
\begin{eqnarray*}
&&\int_{S_r}-2e^{-2u}P^{ijjl}u_l\nu_i d S\\
&=&\int_{S_r}2(n-2)(n-3)e^{4u}u_l(-u_{il}+u_{ss}\delta_{il}-u_i u_l-\frac{n-3}{2}u_s^2\delta_{il})\frac{x_i}{r}d S\\
&=&\int_{S_r}2(n-2)(n-3)e^{4u}u_r\frac{x_l x_i}{r^2}\big[-u_{rr}\frac{x_i x_l}{r^2}-u_r(\frac{\delta_{il}}{r}\!-\!\frac{x_i x_l}{r^3})+(u_{rr}\!+\!\frac{n\!-\!1}{r}u_r)\delta_{il}\\
&&\qquad\qquad\qquad\qquad\qquad\qquad\qquad\;-\frac{x_i x_l}{r^2}u_r^2-\frac{n-3}{2}u_r^2\delta_{il}\big]d S\\
&=&\int_{S_r}2(n-2)(n-3)u_re^{4u}\big[-u_{rr}-\frac{u_r}{r}+\frac{u_r}{r}+u_{rr}+\frac{n-1}{r}u_r-u_r^2-\frac{n-3}{2}u_r^2\big]d S\\
&=&\int_{S_r}2(n-1)(n-2)(n-3)e^{4u}\big[\frac{u_r^2}{r}-\frac 12u_r^3\big]d S.
\end{eqnarray*}
In particular, for metric (\ref{metric}), we have
$$e^{-2u}=(1+\frac{m}{2r^{\frac n2-2}})^{\frac{8}{n-4}},$$
namely,
$$u=-\frac{4}{n-4} {\log}(1+\frac{m}{2r^{\frac n2-2}}),$$
and
\begin{eqnarray*}
u_r&=&-\frac{4}{n-4}\frac{1}{1+\frac{m}{2r^{\frac n2-2}}} \frac{m}{2}(2-\frac n2)r^{1-\frac n2}\\
&=&\frac{m}{1+\frac{m}{2r^{\frac n2-2}}}r^{1-\frac n2}.
\end{eqnarray*}
Therefore, we have
\begin{eqnarray*}
&&\int_{S_r}2(n-1)(n-2)(n-3)e^{4u}\left(\frac{u_r^2}{r}-\frac12u_r^3\right)d S\\
&=&\int_{S_r}\!2(n\!-\!1)(n\!-\!2)(n\!-\!3)\left(1\!+\!\frac{m}{2r^{\frac n2\!-\!2}}\right)^{-\frac{16}{n\!-\!4}}\left(\frac{m^2}{(1\!+\!\frac{m}{2r^{\frac n2\!-\!2}})^2}r^{1-n}\!-\!\frac 12\frac{m^3}{(1\!+\!\frac{m}{2r^{\frac n2\!-\!2}})^3}r^{3\!-\!\frac 32n} \right)d S\\
&=&\int_{S_r}2(n-1)(n-2)(n-3)\left(1+\frac{m}{2r^{\frac n2-2}}\right)^{\frac{-2(n+4)}{n-4}} m^2 r^{1-n} d S\\
&&-\int_{S_r}(n-1)(n-2)(n-3)\left(1+\frac{m}{2r^{\frac n2-2}}\right)^{\frac{-4-3n}{n-4}} m^3 r^{3-\frac 32n} d S.
\end{eqnarray*}
When $n\geq 5,$ the last term converges to  $0$ as $r\rightarrow\infty$.
Hence we have
\begin{eqnarray*}
&&\lim_{r\rightarrow\infty}\int_{S_r}\frac{1}{2(n-1)(n-2)(n-3)\omega_{n-1}}P^{ijkl}\partial_l g_{jk}\nu_id S\\
&=&\lim_{r\rightarrow\infty}\int_{S_r}\frac{1}{\omega_{n-1}}\left(1+\frac{m}{2r^{\frac n2-2}}\right)^{\frac{-2(n+4)}{n-4}} m^2 r^{1-n}d S\\
&=&m^2.
\end{eqnarray*}
Namely, the GBC mass $m_2$ of  metric (\ref{metric}) is exactly equal to $m^2$ as claimed.

One interesting byproduct of the above computation is the following positivity result.
\begin{prop}\label{spherical symmetric}
Suppose $(\cM^n,g), (n\geq 5)$ is asymptotically flat of decay order $\tau>\frac{n-4}{3}$
and $(\cM^n,g)$ is spherically symmetric i.e. $g=e^{-2u(r)}\delta$. Moreover, if $L_2$ is integrable on $(\cM^n,g)$,
then
$$m_2=\lim_{r\rightarrow\infty}\frac{1}{\omega_{n-1}}\int_{S_r}\frac{u_r^2}{r}dS\geq0.$$
\end{prop}
\begin{proof}
By the above calculation, we have
$$m_2=\lim_{r\rightarrow\infty}\frac{1}{\omega_{n-1}}\int_{S_r}e^{4u}\left(\frac{u_r^2}{r}-\frac 12 u_r^3\right) dS.$$
We claim that under the assumption of the decay order, the limit of the second term vanishes as $r\rightarrow\infty.$
In fact, since $(\cM^n,g=e^{-2u(r)}\delta)$ is asymptotically flat of decay order $\tau$, we have $u_r=O(r^{-1-\tau})$ and  $u=O(r^{-\tau})$.
Combining with the condition of decay order $\tau>\frac{n-4}{3}$, we thus get
$$u_r=o(r^{-\frac{n-1}{3}}),$$ which yields the limit of the second integral vanishes as $r\rightarrow\infty.$ Moreover, $e^{4u}=1+o(1)$. Hence
the desired result follows from the claim.
\end{proof}

\begin{rema}
By this Proposition and  the previous positive mass theorem for graphs, there are no spherically symmetric asymptotically flat smooth functions on ${\mathbb R}^n$
whose graphs have negative  Gauss-Bonnet curvature $L_2$ everywhere.
\end{rema}

\begin{rema}
Given  an asymptotically flat and  spherically symmetric manifold $(\cM^n,g)$,  a direct computation gives the $ADM$ mass
\begin{eqnarray*}
m_{ADM}&\triangleq&\lim_{r\rightarrow\infty}\frac{1}{2(n-1)\omega_{n-1}}\int_{S_r}(g_{ij,i}-g_{ii,j}) \nu_j d S\\
&=&\lim_{r\rightarrow\infty}\frac{1}{\omega_{n-1}}\int_{S_r}e^{-2u}u_r d S.
\end{eqnarray*}
Thus if the manifold is asymptotically flat of decay order $\tau>\frac {n-2}{2}$ and $(\cM^n,g)$ is spherically symmetric, that is, $g=e^{-2u(r)}\delta,$ and if the scalar curvature $R$ is integrable, unlike our case
the ADM mass is not always nonnegative.
\end{rema}

\begin{rema}\label{Schgraph}
It is well-known that the Schwarzschild metrics have zero scalar curvature and in view of analogy, the Gauss-Bonnet curvature  $L_2$ with respect to metric (\ref{metric}) is equal to $0$. One can check from the above calculation. Moreover, the metric in Example \ref{Schwarzschild} can be realized as a graph with the induced metric from the Euclidean space ${\mathbb R} ^{n+1}$. For example, when $n=5$, metric (\ref{metric}), in its equivalent  form,
 $$(\cM^5,g)=\big({\mathbb R}^5\setminus\{0\},\big(1+\frac {{m}}{2r^{\frac{1}2}}\big)^{8}g_{{\mathbb R}^5}\big)$$ can be isometrically embedded as a rotating parabola  in $\{(x_1,x_2,x_3,x_4,x_5,w)\subset{\mathbb R}^6\}$. The outer end of metric (\ref{metric}) containing the infinity
is the graph of the spherically symmetric function $f:{\mathbb R}^5\setminus B_{4m^2}(0)\rightarrow{\mathbb R}$ given by $f(r)=2r^{\frac 12}\sqrt{8m(r^{\frac 12}-2m)}-\frac{1}{6m}[8m(r^{\frac 12}-2m)]^{\frac 32},$ where $r=|(x_1,x_2,x_3,x_4,x_5)|.$
\end{rema}

\begin{lemm}\label{sigma}
Assume  $(\cM^n,g) (n\geq 5)$ is a $n$-dimensional submanifold in ${\mathbb R} ^{n+1}$.
Then
$$L_2=24H_4:=24\sum_{i<j<k<l}\lambda_i\lambda_j\lambda_k\lambda_l,$$
where $(\lambda_1,\dots,\lambda_n)$ is the set of eigenvalues of the second fundamental form $A$.
\end{lemm}
\begin{proof}
We recall the Gauss equation
$$R^{ijkl}=A^{ik}A^{jl}-A^{il}A^{jk}.$$
Thus the desired result yields from a direct calculation.
\end{proof}

As a consequence, we obtain the following result.
\begin{prop}
Let $f:{\mathbb R}^n\to {\mathbb R}$ be a smooth radial function in Definition \ref{f}.   Then  the second fundamental form $A$ has $n-1$ eigenvalues $\frac{f_r}{r\sqrt{1+f_r^2}}$ and one eigenvalue $\frac{f_{rr}}{(\sqrt{1+f_r^2})^3}$. Hence
$$
L_2=24\left(\frac{(n-1)!}{(n-5)!\;4!\;}\frac{f_r^4}{r^4(1+f_r^2)^2}+\frac{(n-1)!}{(n-4)!\;3!\;}\frac{f_r^3f_{rr}}{r^3(1+f_r^2)^3}\right),
$$
so that if $L_2$ is integrable, then
$$
m_2=\lim_{r\to+\infty}\frac{1}{w_{n-1}}\int_{S_r}\frac{r^{n-4}f_r^4(r)}{4
}\ge 0.
$$
\end{prop}

\section{Applications to the ADM mass}

In this section, we  provide an interesting application of our study presented above.

As mentioned in the introduction, our paper is motivated by the Einstien-Gauss-Bonnet theory, in which  one studies  the following action
\begin{equation}\label{xeq0.1}
R+\Lambda+\alpha L_2(g)
\end{equation}
on $n$-dimensional manifolds with $n\geq 4$.
For this theory, there is a mass introduced by Deser-Tekin \cite{DT1, DT2}. The precise definition is given in the following

\begin{defi}[Einstein-Gauss-Bonnet Mass]\label{GBC} Let $n\ge 4$.
Suppose $(\cM^n,g)$ is an asymptotically flat manifold of decay order
\begin{equation*}
\tau>\frac{n-2}{2},
\end{equation*}
 and $R+\alpha L_2$  is integrable on $(\cM^n,g)$. We define the Einstein-Gauss-Bonnet mass-energy by
\begin{equation}\label{EGB_mass}
m_{EGB}= \frac 1{2(n-1)\omega_{n-1}} \lim_{r\rightarrow\infty}\int_{S_r}\{(g_{ij,j}-g_{jj,i})+2\alpha P^{ijkl}g_{jk,l}\} \nu_i dS,
\end{equation}
where $\nu$ is the outward unit normal vector to $S_r$, $dS$ is the area element of $S_r$ and the tensor $P$ is defined as in (\ref{P}).
\end{defi}

As for the GBC mass $m_{GBC}$, one can also show that this Einstein-Gauss-Bonnet mass  $m_{EGB}$ is a well-defined  geometric invariant, provided that
$R+\alpha L_2$ is integrable in $(\cM^n,g)$ and the decay satisfies
\begin{equation} \label{tau_fast}
\tau>\frac{n-2}{2}.
\end{equation}
This decay condition \eqref{tau_fast} is necessary for the well-definedness of the first part in \eqref{EGB_mass}. However, as we already proved,
this decay condition makes the second part in \eqref{EGB_mass} vanishing. This simple observation implies that
\beq\label{id}
m_{EGB}=m_{ADM}.
\eeq
Namely, in the Einstein-Gauss-Bonnet theory we obtain a mass which is the same as the ADM mass. This triviality, together with our work, implies a
new positive mass theorem for the ADM mass.

\begin{theo}[Positive Mass Theorem]\label{PMT}
{
Let $(\cM^n,g)=(\mathbb{R}^n,\delta+df\otimes df)$ be the graph of a smooth asymptotically flat function $f:\mathbb{R}^n\rightarrow \mathbb{R}$ such that
decay order $\tau>\frac {n-2} 2$ and $R+\alpha L_2$  is integrable on $(\cM^n,g)$.
Then we have \begin{equation}\label{eq2.1}
m_{ADM}=m_{EGB}=\frac {1}{2(n-1)\omega_{n-1}}\int_{\cM^n}\bigg(R+\alpha L_2\bigg)\frac{1}{\sqrt{1+|\nabla_{\delta}f|^2}}dV_g,\end{equation}
where $\alpha$ is a constant. In particular,
\begin{equation}\label{PC}
R+\alpha L_2\geq 0,
\end{equation}
 yields $$m_{ADM}\geq 0.$$
}
\end{theo}
The proof follows from the above work and the work of Lam \cite{Lam} for the case $k=1$.
Hence, we have  a non-negative mass result under a different condition $R+\alpha L_2\geq 0$ other than the usual dominant condition
$R\ge 0$.


We also have a Penrose type inequality under  condition (\ref{PC}), at least for the graphs.

\begin{theo}\label{Penrose}
Let $\Omega$ be a bounded open set in $\mathbb{R}^n$ and $\Sigma=\partial\Omega.$ If $f:\mathbb{R}^n\setminus\Omega\rightarrow \mathbb{R}$ is a smooth
asymptotically flat function such that each connected component of $\Sigma$ is in a level set of $f$ and $|\nabla f(x)|\rightarrow\infty$ as
$x\rightarrow\Sigma$.
Assume further that  $(\cM^n,g)=(\mathbb{R}^n,\delta+df\otimes df)$ 
has decay order $\tau>\frac {n-2} 2$ and $R+\alpha L_2$  is integrable on $(\cM^n,g)$. Let $H_1$ and $H_3$ denote the mean curvature and the $3$-th mean curvature of $\Sigma$ induced by the Euclidean metric respectively. Then
\begin{eqnarray*}
m_{ADM}&=&\frac {1}{2(n-1)\omega_{n-1}}\int_{\cM^n}\bigg(R+\alpha L_2\bigg)\frac{1}{\sqrt{1+|\nabla_{\delta}f|^2}}dV_g\nonumber\\
&&+\frac {1}{2(n-1)\omega_{n-1}}\int_{\Sigma}\bigg(H_1+2\alpha\cdot 3H_3\bigg) dS.
\end{eqnarray*}
If we assume further that each connected component of  $\Sigma$ is convex and $\alpha \ge 0$, we have
\begin{eqnarray*}
m_{ADM}&\geq&\frac {1}{2(n-1)\omega_{n-1}}\int_{\cM^n}\bigg(R+\alpha L_2\bigg)\frac{1}{\sqrt{1+|\nabla_{\delta}f|^2}}dV_g\nonumber\\
&&+\frac{1}{2}\bigg(\frac{|\Sigma|}{\omega_{n-1}}\bigg)^{\frac{n-2}{n-1}}+\frac{\alpha}{2}(n-2)(n-3)\bigg(\frac{|\Sigma|}{\omega_{n-1}}\bigg)^{\frac{n-4}{n-1}}.
\end{eqnarray*}
In particular,
$R+\alpha L_2\geq 0$ implies
\beq\label{new_P}
m_{ADM}\geq\frac{1}{2}\bigg(\frac{|\Sigma|}{\omega_{n-1}}\bigg)^{\frac{n-2}{n-1}}+\frac{\alpha}{2}(n-2)(n-3)\bigg(\frac{|\Sigma|}{\omega_{n-1}}\bigg)^{\frac{n-4}{n-1}}.\eeq
\end{theo}

\vspace{5mm}
\begin{proof} The proof follows from the proof given in \cite{Lam} and in the previous section. \end{proof} 

This Penrose type inequality is also optimal. i.e. equality in \eqref{new_P} holds at the following Schwarzschild-type metric  (See \cite{BD} and \cite{CO})
\begin{equation}\label{xmetric}
g=\bigg(1+\frac{r^2}{\tilde{\alpha}}\bigg\{1-\sqrt{1+\frac{4\tilde{\alpha}m}{r^n}}\bigg\}\bigg)^{-1}dr^2+r^2d\Theta^2,
\end{equation}
where $\tilde{\alpha}=2(n-2)(n-3)\alpha$ and  $d\Theta^2$ is the standard metric on ${\mathbb S}^{n-1}$.
This metric can be also represented as a graph which satisfies all conditions in Theorem \ref{Penrose}.
Its  horizon $\Sigma$ is the surface $\{r=r_0\}$ with $r_0$ being the solution of
$$1+\frac{r^2}{\tilde{\alpha}}\bigg(1-\sqrt{1+\frac{4\tilde{\alpha}m}{r^n}}\bigg)=0,$$
which yields
\begin{equation}\label{add_z}
\begin{array}{rcl}
m&=&\ds\vs \frac 12r_0^{n-2}+\frac {\tilde{\alpha}}{4} r_0^{n-4}\\
&=&\ds \frac 12\bigg(\frac{|\Sigma|}{\omega_{n-1}}\bigg)^{\frac{n-2}{n-1}}+\frac {\alpha}{2}(n-2)(n-3)\bigg(\frac{|\Sigma|}{\omega_{n-1}}\bigg)^{\frac{n-4}{n-1}}.
\end{array}\end{equation}
A direct computation gives
\begin{eqnarray*}
R(g)&=&
\frac{8n(n-1)\tilde \alpha m^2}{r^{2n}\big(1+\sqrt{1+\frac{4\tilde \alpha m}{r^n}}\big)^2\sqrt{1+\frac{4\tilde\alpha m}{r^n}}}.
\end{eqnarray*}
Recall $\tilde{\alpha}=2(n-2)(n-3)\alpha$, hence $R>0$ if $\alpha>0$ and $R<0$ if $\alpha <0$.
One can also check  that the  metric (\ref{xmetric}) has $R+\alpha L_2=0$. In fact this metric is a black hole
solution in the Einstein Gauss-Bonnet theory, see \cite{BD} and \cite{Cai}.
Moreover, we have

\begin{prop}
The EGB mass of the above metric is equal to $m$, i.e.,
\[m_{EGB}=m_{ADM}=m=\frac 12\bigg(\frac{|\Sigma|}{\omega_{n-1}}\bigg)^{\frac{n-2}{n-1}}+\frac {\alpha}{2}(n-2)(n-3)\bigg(\frac{|\Sigma|}{\omega_{n-1}}\bigg)^{\frac{n-4}{n-1}}.\]
\end{prop}
\begin{proof}
By a reduction, metric (\ref{xmetric}) becomes
$$g=\bigg(1-\frac{2m}{r^{n-2}}\frac{2}{1+\sqrt{1+\frac{4\tilde{\alpha}m}{r^n}}}\bigg)^{-1}dr^2+r^2d\Theta^2,$$
which is an asymptotically flat metric of decay order $\tau=n-2.$ Thus $m_{EGB}(g)=m_{ADM}(g)$.
Comparing with the ordinary Schwarzschild solution
\begin{equation}\label{Sch}
g_{\rm Sch}=\bigg(1-\frac{2m}{r^{n-2}}\bigg)^{-1}dr^2+r^2d\Theta^2,
\end{equation}
we know that the difference is $$g-g_{Sch}=O(r^{-2n-2}),$$
which has   a decay fast enough and hence it does not contribute to the ADM mass.
Therefore we have
$$m_{ADM}(g)=m_{ADM}(g_{\rm Sch})=m,$$ which follows from the fact that the ADM mass of the Schwarzschild metric (\ref{Sch}) is exactly equal to $m$. One can also  compute it directly. The rest follows from \eqref{add_z}.
\end{proof}
Therefore, the Penrose inequalities in Theorem \ref{Penrose} are optimal. This is a new  Penrose type inequality for the ADM mass
under the assumption that $R+\alpha L_2\ge 0.$

\section{Generalization, Problems and conjectures}

First of all, we can generalize our results to $k<n/2$.
In the definition of the GBC mass, in the proof of its geometric invariance and in the proof of
the positive mass theorem for graphs over $\R^n$ one can see that the crucial things  are
the divergence-free property and the symmetry (and also anti-symmetry) of the tensor $P$. Hence,  with a completely same argument, we can
define
a mass for $L_k$-curvature for any $k<n/2$. For general $L_k$ curvature, the corresponding $P_{(k)}$ curvature is
\begin{equation}\label{Pk}
P_{(k)}^{st lm}:=\frac{1}{2^k}\d^{i_1i_2\cdots i_{2k-3}i_{2k-2}st}
_{j_1j_2\cdots  j_{2k-3}j_{2k-2} j_{2k-1}j_{2k}}{R_{i_1i_2}}^{j_1j_2}\cdots
{R_{i_{2k-3}i_{2k-2}}}^{j_{2k-3}j_{2k-2}}g^{j_{2k-1}l}g^{j_{2k}m}.
\end{equation}
We can define a mass for $1\le k<n/2$ by
\begin{equation}\label{mk}
m_k=c(n,k)\lim_{r\rightarrow\infty}\int_{S_r}P_{(k)}^{ijml}\partial_l g_{jm}\nu_{i}dS,\end{equation}
with a dimensional constant
$$c(n,k)=\frac{(n-2k)!}{2^{k-1}(n-1)!\;\omega_{n-1}}.$$
 This constant can be determined by computing Example \ref{Schwarzschild} such that the mass $m_k=m^k$.
Remark that for even $k$, the Gauss-Bonnet-Chern mass $m_k$ of the metric $g^{(k)}_{\rm Sch}$ is positive even for negative $m$.
One can  check that $P_{(k)}$ has the same  divergence-free property  and the same symmetry (and also anti-symmetry) as the tensor $P$.
It is clear that ${P_{(2)}}=P$ and since $R=\frac12(g^{il}g^{jm}-g^{im}g^{jl})R_{ijlm}$, we have
\[P_{(1)}^{ijlm}=\frac 12 (g^{il}g^{jm}-g^{im}g^{jl}).\]
If we use this tensor $P_{(1)}$ to define a mass, it is just the ADM mass, with a slightly different, and certainly equivalent form
\[\frac1{2(n-1)\omega_{n-1}}\lim_{r\to \infty}\int_{S_r}(g^{il}\partial_jg_{jl}-g^{jl}\partial_ig_{jl})\nu_i dS.\]
However, it is interesting to see that with this form one can directly compute to obtain for the ADM mass $m_1$ that
\[m_1=\frac 1{2(n-1)\omega_{n-1}}\lim_{r\to \infty}\int_{S_r}\frac 1{1+|\nabla f|^2}(f_{ii}f_i-f_{ij}f_i)\nu_jdS\]
without using a trick in the proof of Theorem 5 in \cite{Lam} by adding a factor $1\slash  (1+|\nabla f|^2)$. This is the reason why we do not need to use this trick in our proof
of Theorem \ref{mainthm2}.\\

With the same crucial property of $P_{(k)}$, we can show the positive mass theorem  and the Penrose inequality for $m_k$ in the case of graphs, provided that the decay order satisfies
\[\tau>\frac{n-2k}{k+1}.\]

Moreover, using the Gauss-Bonnet curvature
$L_2$ (and also $L_k$ $(k<n/2)$) we can also introduce a GBC mass $m_2^H$ for asymptotically hyperbolic manifolds in \cite{GWW2}.
The study of  the ADM mass for asymptotically hyperbolic manifolds was initiated  by X. Wang \cite{Wang} and  Chru\'sciel-Herzlich \cite{CH}. See also \cite{ZhangX}. There are many interesting generalizations.
 Here we just mention
the recent work of Dahl-Gicquaud-Sakovich  \cite{DGS} and  Lima-Gir\~ao \cite{dLG3} for  asymptotically hyperbolic graphs.
In \cite{GWW2} we obtained a positive mass theorem for $m_2^H$  for asymptotically hyperbolic graphs
provided $$ L_2(g)+2(n-2)(n-3)R(g)\ge L_2(g_{{\mathbb H}^n})+2(n-2)(n-3)R(g_{{\mathbb H}^n}),$$
where $g_{{\mathbb H}^n}$ is the standard hyperbolic metric. A Penrose type inequality was also obtained.

\

 There are many interesting problems we would like to consider for the new mass.

First all, it would be an interesting problem to consider the relationship between the Gauss-Bonnet-Chern mass and the Gauss-Bonnet-Chern theorem. The mass defined in Definition \ref{m2}
can be also defined for $n=4$. In this case, the decay order (\ref{eq0.1}) needs
\[\tau>0.\]
However, this decay condition forces the $m_2$ mass vanishing. Nevertheless, it is interesting to use it to consider  asymptotically conical manifolds
 in dimension 4.
There are interesting results about the Gauss-Bonnet-Chern theorem on higher dimensional, noncompact manifolds using $Q$-curvature initiated by Chang-Qing-Yang \cite{CQY}.

It would be interesting to ask if Theorem \ref{mainthm2} is true for general asymptotically flat manifolds

\

\noindent \textbf {Problem 1.}  {\it Is the GBC mass $m_2$ nonnegative for an asymptotically flat manifold with $\tau>\frac {n-4}3$  and $L_2(g)\ge 0$.}

\

We conjecture that this is true, at least under an additional condition that the scalar curvature $R$ is nonnegative.
The Schwarzschild metric (\ref{metric}) has $L_2=0$ and $R>0$. It would be already interesting
if one can show its nonnegativity  for locally conformally flat manifolds. We can generalize Theorem \ref{mainthm2} to show the nonnegativity
of $m_2$
for a class of  hypersurfaces in a manifold with a certain product structure.   This is related to the recent work of  Lima \cite{dLG1} and \cite{HW}
for the ADM mass. This, together with a positivity result for conformal flat metrics in ${\mathbb R}^n$, will be presented in a forthcoming paper
\cite{GWW3}.

The GBC mass $m_2$ is closely related to  the $\sigma_2$ Yamabe problem.  With a suitable definition of the Green function for
the $\sigma_2$ Yamabe problem, one would like to ask the existence of the Green function and its expansion. The leading term of the regular part in the expansion of the Green function  should  be closely related to the
mass $m_2$. Metric (\ref{metric}) does provide such an example.
For the relationship between the ADM mass, the expansion of the ordinary Green function and
the resolution of the ordinary Yamabe problem,
see \cite{Schoen} and \cite{LP}.

\

\noindent \textbf {Problem 2.} {\it Is there the rigidity result?}

\

Namely, is it true if  $m_2=0$, then $\cM=\R^n$?
Proposition 6.2 and Proposition 6.7 show that the rigidity result holds for two (very) special classes
of manifolds,   one is the class of  spherically symmetry and conformally  flat manifolds and another is the class of spherically symmetry graphs.  For these two classes of manifolds the  mass vanishes implies that the  manifold is isometric to the Euclidean space. Therefore,
 it is natural to conjecture that  the rigidity  result holds, at least under additional condition
 that its scalar curvature is nonnegative. This is a difficult problem, even in the case of  the asymptotically flat graphs.
In this case, it is in fact a Bernstein type problem. Namely, is there a non-constant function
satisfying
\begin{eqnarray}\label{addx2}
2L_2=P^{ijkl}\left(\frac{f_{ik}f_{jl}-f_{il}f_{jk}}{1+|\nabla f|^2}\right)=0,
\end{eqnarray}
under the decay conditions given in Definition \ref{f}? For the related results on the rigidity result of the ADM mass for graphs, see \cite{HW} and \cite{dLG2}.

\

\noindent \textbf {Problem 3.} {\it Does the Penrose inequality for the GBC mass hold on
general asymptotically flat manifolds?}

\

Theorem \ref{mainthm3} provides three inequalities. Two of them involve only intrinsic invariants, which
we consider as the generalized forms of the ordinary Penrose inequality \cite{HI} and \cite{Bray}.
Comparing with the ordinary Penrose inequality, we conjecture
$$m_2\geq \frac{1}{4}\bigg(\frac{|\Sigma|}{\omega_{n-1}}\bigg)^{\frac{n-4}{n-1}},$$
if $\Sigma$ is an area outer minimizing horizon and
$$m_2\geq \frac{1}{4}\bigg(\frac{\int_\Sigma R}{(n-1)(n-2)\omega_{n-1}}\bigg)^{\frac{n-4}{n-3}},$$
if $\Sigma$ is an outer minimizing horizon for the functional
\[\int_{\Sigma} R,\]
whose Euler-Lagrange equation is
\[E^1_{ij}B^{ij}=0.\]
 Here $E^1$ is the ordinary Einstein tensor.
For $m_k$ ($k<n/2$), one should have $k$ inequalities relating  to $m_k$ with
\[\int_{\Sigma}L_{j}(g)dv(g), \quad j=0,1,2,\cdots, k-1.\]
These functionals were considered in \cite{LiA} and \cite{Labbi}.  Note that $L_1=R$.

\

\noindent{\it Acknowledgment.} The GBC mass and results presented in this paper were reported in the conference  ``Geometric PDEs" in trimester ``Conformal and K\"ahler Geometry" in IHP, Paris  held on
Nov. 5-- Nov. 9, 2012.
Afterwards, Yanyan Li  informed us that  Luc Nguyen and  he also defined a mass by  using an invariant
which agrees with the $\sigma_k$ curvature when the manifold
is locally conformally flat and the mass is non-negative
under the assumption that the $\sigma_k $-curvature (or some
other curvature invariant) is non-negative, if in addition
that the manifold is locally conformally flat and the
$\sigma_k$-curvature is zero near infinity, together with a
rigidity in that case.
We would like to thank K. Akutagawa, B. Ammann, A. Chang,  P. Guan,  E. Kuwert, Yanyan Li, J. Qing, J. Viaclovsky and P. Yang for their interest and  questions,
which improve the writing of this paper.

\end{document}